\documentclass[final]{lmsmod}
\usepackage{pdfsync}

\usepackage{amsmath,amsfonts,latexsym,amscd,amssymb,enumerate}

\usepackage{hyperref}

\newtheorem{theorem}{Theorem}[section]
\newtheorem{lemma}[theorem]{Lemma}
\newtheorem{corollary}[theorem]{Corollary}
\newtheorem{proposition}[theorem]{Proposition}

\newtheorem{definition}[theorem]{Definition}
\newtheorem{example}[theorem]{Example}
\newtheorem{notation}[theorem]{Notation}

\newtheorem{set-up}[theorem]{Geometric set-up}
\newtheorem{remark}[theorem]{Remark}

\newcommand{\relC}[2]{C^*(#2\subset {#1})} 
\newcommand{\relD}[2]{D^*(#2\subset{#1})} 

\newcommand{\reals}{\mathbb{R}}
\newcommand{\complexs}{\mathbb{C}}
\newcommand{\naturals}{\mathbb{N}}

\newcommand{\K}{\mathbb{K}}

\DeclareMathOperator{\id}{id}
\newcommand{\boundary}[1]{\partial#1}

\newcommand{\abs}[1]{\left\lvert#1\right\rvert} 
\newcommand{\norm}[1]{\left\lVert#1\right\rVert}

\newcommand{\tensor}{\otimes}
\newcommand{\into}{\hookrightarrow}

\newcommand{\iso}{\cong}

\newcommand{\disjointunion}{\amalg}

\DeclareMathOperator{\Ad}{Ad}
\DeclareMathOperator{\cyl}{cyl}

\DeclareMathOperator{\supp}{supp}   

\DeclareMathOperator{\spin}{spin}

\DeclareMathOperator{\Pos}{Pos}

\DeclareMathOperator{\ind}{ind}
\DeclareMathOperator{\coarseind}{Ind}
\DeclareMathOperator{\Ind}{Ind}

\DeclareMathOperator*{\dirlim}{dirlim}

\newcommand{\forget}[1]{}

\def  \nuint {\raise10pt\hbox{$\nu$}\kern-6pt\int}

\newcommand\E{\mathcal E}
\newcommand\Q{\mathcal Q}

\def \Sp {{\cal S}}

\def \Ch {{\rm Ch}}
\def\Id{{\rm Id}}

\newcommand\ha{\frac12}

\renewcommand\Im{\operatorname{Im}}

\newcommand\D{\mathcal D}
\newcommand\Di{D\kern-6pt/}
\newcommand\cDi{{\mathcal D}\kern-6pt/}
\newcommand\spi{S\kern-6pt/}
\newcommand \cspi{\Sp\kern-6pt/}

\newcommand\CC{\mathbb C}

\def \cal {\mathcal}

\def \K {{\cal K}}

\def \BB {\mathbb B}

\newcommand\NN{\mathbb N}

\newcommand\RR{\mathbb R}

\newcommand\KK{\mathbb K}

\newcommand\pa{\partial}

 \usepackage{color}
\definecolor{darkgreen}{cmyk}{1,0,1,.2}
\definecolor{m}{rgb}{1,0.1,1}

{\catcode`@=11\global\let\c@equation=\c@theorem}
\renewcommand{\theequation}{\thetheorem}

\renewcommand{\theenumi}{(\arabic{enumi})}

\allowdisplaybreaks[2]

\begin{document}
\pagestyle{myheadings}
\markboth{Paolo Piazza and Thomas Schick}{Rho-index theorems and Stolz'
  psc sequence}


\title[Rho-classes,  index theory and Stolz' psc sequence]{Rho-classes,  index theory and Stolz' positive
  scalar curvature sequence}

\author{Paolo Piazza and Thomas Schick}

\extraline{The first author was partially funded by {\it Ministero  dell'Istruzione, dell'Universit\`a e della Ricerca} through
the project {\it Spazi di moduli e teoria di Lie} and as a guest for several week-long visits in
 G\"ottingen by  the Courant Research
 Center ``Higher order 
 structures in Mathematics'' within the German initiative of excellence. The
 second author was partially funded by the Courant Research Center ``Higher order structures in Mathematics''
within the German initiative of excellence, and partially funded by Sapienza Universit\`a di Roma
for a three-months
visiting professorship in Rome}

\maketitle
\begin{abstract}
  {In this paper, we study the space of metrics of positive scalar
    curvature using methods from coarse geometry.

  Given a closed spin manifold $M$ with fundamental group $\Gamma$, Stephan
  Stolz introduced the positive scalar curvature exact 
  sequence, in analogy to Wall's surgery exact sequence in topology. It
  calculates a structure
  group of metrics of positive scalar curvature on $M$ (the object we want to
  understand) in terms of spin-bordism of $B\Gamma$ (the classifying space of
  $\Gamma$) and a further group $R^{\spin} (\Gamma)$.

  Higson and Roe introduced a K-theory exact sequence $\to
  K_*(B\Gamma)\xrightarrow{\alpha} K_*(C^*_\Gamma)\xrightarrow{j} K_{*+1}(D^*_\Gamma)\to$ in coarse geometry which
  contains the Baum-Connes assembly map $\alpha$, with $K_*(D^*_\Gamma)$ 
  canonically associated to $\Gamma$. The K-theory groups in
  question are the home of interesting index invariants and secondary
  invariants, in particular the rho-class  
  $\rho_\Gamma (g)\in
  K_*(D^*_\Gamma)$ of a metric of positive scalar curvature.

  One of our main results is the construction of a map from
  the Stolz exact sequence to the Higson-Roe exact sequence (commuting with all
  arrows), using coarse index theory throughout. This theorem complements the results
  of Higson and Roe in \cite{higson-roeI,higson-roeII,higson-roeIII} where they
  show that it is indeed possible to map the surgery exact sequence in topology to their sequence 
  $\to
  K_*(B\Gamma)\xrightarrow{\alpha} K_*(C^*_\Gamma)\xrightarrow{j} K_{*+1}(D^*_\Gamma)\to$.
     
\smallskip
  Our main tool is an index theorem of Atiyah-Patodi-Singer type, which
  is of independent 
  interest. Here, assume that $Y$ is a compact spin manifold with boundary,
  with  
  a Riemannian metric $g$ which is of  positive scalar curvature when restricted to the boundary
  (and $\pi_1 (Y)=\Gamma$). Because the Dirac
  operator on the boundary is invertible, one constructs an APS-index
   $\coarseind_\Gamma (Y)\in K_* (C^*_\Gamma)$. This can be pushed forward to
  $j_* (\coarseind_\Gamma (Y))\in K_* (D^*_\Gamma)$ (corresponding to the
  ``delocalized   part'' of 
  the index). We then prove a delocalized APS-index
  theorem, equating this 
  class to the rho-class of the boundary
 $   j_* (\coarseind_\Gamma (Z)) = \rho_\Gamma(g_{\pa Z}) \in K_* (D^*_\Gamma)$.

 As a companion to this, 
we prove a secondary
 partitioned manifold index theorem. Given a (non-compact) spin manifold $W$ with positive scalar
 curvature metric $g$, with a free and discrete
  isometric action by a  group $\Gamma$ and a $\Gamma$-invariant cocompact 
 partitioning hypersurface $M$, one can use a ``partitioned
  manifold construction'' in order to obtain the partitioned manifold
  $\rho$-class  $\rho_\Gamma^{{\rm pm}}(g)\in K(D^*_\Gamma)$. Assume in addition that
  $M$ has a tubular neighborhood where the metric is a product
  $g=g_M+dt^2$. Then we prove the partitioned manifold $\rho$-class  theorem
 $   \rho_\Gamma(g_M) = \rho_\Gamma^{{\rm pm}}(g) \in K_* (D^*_\Gamma)$.
We use this secondary partitioned manifold index theorem to distinguish isotopy
classes of positive scalar curvature on $W$.
\smallskip

}  
\end{abstract}
\tableofcontents


\section{Introduction and main results}\label{sec:intro}


\subsection{Basics on coarse geometry and coarse index theory}
\label{sec:basics}

We start by recalling the basic constructions of coarse geometry, their
associated $C^*$-algebras and K-theory as used in the paper. We shall freely
use concepts and results from \cite{hr-novikov} and \cite{roe-cbms}.


\begin{definition}
  Let $X$ be a complete  Riemannian manifold of positive dimension, $C_c(X)$
  the compactly supported continuous functions with values in $\complexs$ and
  $C_0(X)$ its sup-norm closure, the continuous functions vanishing at
  infinity.

  Let $E\to X$ be a
  Hermitean vector bundle. We consider $H:=L^2(E)$ and $H':=L^2 (E)\tensor
  l^2(\naturals)$. These are so-called \emph{adequate} $X$-modules, which means that $H$ is a Hilbert space with a
  $C^*$-homomorphism $C_0(X)\to B(H)$, given here by pointwise multiplication, and if $0\ne f\in C_0(X)$ then it does
  not act as compact operator, and that $C_0(X)H$ is dense in $H$.
 We have a
canonical isometry $u\colon H\to H'$ mapping 
into the first direct summand of $l^2(\naturals)$. Using this, we map
an operator $A$ on $H$ to the operator $uAu^*$ on $H'$. We will implicitly do
this throughout the paper and this way consider the operators on $H$
as operators on $H'$, without explicitly mentioning it.

  \begin{itemize}
  \item $D_c^*(X,H)$ is defined to be the algebra of bounded operators $T$ on
    $L^2(E)\tensor l^2(\naturals)$
    with the following properties:
    \begin{enumerate}
    \item $T$ has finite propagation, which means that there is an $R>0$ such
      that for each $s\in L^2(E)$ and
      for each $x\in\supp(Ts)$, $d(x,\supp(s))<R$.
    \item $T$ is \emph{pseudo-local}: for each $\phi\in C_c(X)$, the commutator
      $[T,\phi]$ is compact.
    \end{enumerate}
 $D^*(X,H)$ is defined to be the norm closure of $D_c^*(X,H)$.
   \item $C^*_c(X,H)$ is defined to be the subalgebra of $D^*_c(X,H)$ of
     operators which are in addition
    \emph{locally compact}, i.e.~$T\phi$ and $\phi T$ are compact for
       each 
       $\phi\in C_c(X)$.
 $C^*(X,H)$ is the $C^*$-closure of $C_c^*(X,H)$. This is the \emph{Roe
       algebra} of $X$.
  \end{itemize}
The definition generalizes to an arbitrary proper metric space $X$;
$L^2(E)$ then has to be replaced by an abstract adequate $C_0(X)$-module.
\end{definition}

    There are natural functoriality properties
 that we recall from
    \cite{hr-novikov,hry-mv,roe-cbms}:
    \begin{definition} A map
     $f\colon X\to Y$ between proper metric spaces is a \emph{coarse map} if
     for each $R>0$ there is $S>0$ such that the image under $f$ of every
     $R$-ball is contained in an $S$-ball and, moreover, the inverse image of
     every bounded set is bounded.
    \end{definition}

We will not prove the following functoriality results of Higson and Roe, but
we recall, after the Proposition, the relevant construction which we are going to use.
    \begin{proposition}\label{prop:functoriality}
If $f\colon X\to Y$ is a continuous coarse map and $H_X,H_Y$ are adequate
$C_0(X)$ or $C_0(Y)$-modules, respectively, then $f$ induces a
non-canonical, but with suitable choices
functorial homomorphism $f_*\colon D^*(X,H_X)\to D^*(Y,H_Y)$ which maps
$C^*(X,H_X)$ to $C^*(Y,H_Y)$. The induced map in K-theory is \emph{canonical}.
    \end{proposition}

    \begin{definition}
Applying Proposition \ref{prop:functoriality} to $\id\colon X\to X$ we observe
that $C^*(X,H_X)$ and $D^*(X,H_X)$ depend only mildly on the adequate module
$H_X$, and that their K-theory is independent of this choice. We follow the
custom of \cite{roe-cbms} and drop $H_X$ from the notation, writing simply
$C^*(X)$ or $D^*(X)$ instead of $C^*(X,H_X)$, $D^*(X,H_X)$.
    \end{definition}

For the construction of $f_*$ of Proposition \ref{prop:functoriality}, we need
the following concepts:
    \begin{definition}
    Let $H_X$ and $H_Y$ two
    adequate modules and let $f\colon X\to Y$ be a coarse map.
    We say that an isometric embedding  $W\colon H_X\to H_Y$
    \emph{covers $f$ in the $C^*$-sense} if $W$ is the norm-limit 
    of linear maps $V$ satisfying the following condition:
    \begin{equation}\label{condition-c}
     \exists \;R>0 \text{ such that } \phi V \psi =0 \text{ if }d({\rm Supp} (\phi),f({\rm Supp} (\psi)))>R,
    \text{ for } \phi\in C_c (Y) \text{ and }\psi\in C_c (X).
    \end{equation} 
    \end{definition}    

    Given a coarse map it is always possible to find such a $W$.
    Then the map $\Ad (W) (T):= W T W^*$, from the bounded
    operators on $H_X$ to the bounded 
    operators of $H_Y$, sends $C^*(X,H_X)$ to $C^* (Y,H_Y)$ and we define
    $f_*:=\Ad(W)\colon 
    C^*(X,H_X)\to C^*(Y,H_Y)$. The induced map in K-theory
     is independent of the choice of $W$, see \cite[Lemma 3]{hry-mv}.
    Moreover, by \cite{hr-trans} the functor $K_* (C^* (X))$ 
    is a coarse homotopy invariant.
    
    Regarding $D^*(X,H)$ we have the following.
     
    \begin{definition}\label{def:cover-D*-sense}
 Let $f\colon X\to Y$ be a \emph{continuous} coarse
    map, let $H_X$, $H_Y$ be two adequate modules. We shall say that an
    isometry  $W\colon H_X\to H_Y$ \emph{covers $f$ in the
  $D^*$-sense}\footnote{in \cite[Definition 2.4]{Siegel}, the same property is denoted ``$W$
  covers $f$ topologically''} if $W$ is the norm-limit 
    of bounded maps $V$ satisfying the following two conditions:
\begin{itemize}    
\item there is an $R>0$ such that $\phi V \psi =0$ if $d({\rm Supp} (\phi),f({\rm Supp} (\psi)))>R$,
    for $\phi\in C_c (Y)$ and $\psi\in C_c (X)$;
    \item  $\phi V - V (\phi\circ f)$
    is compact for each $\phi\in C_0 (Y)$.
    \end{itemize}
    \end{definition}    
   For such a $W$ one proves that $\Ad (W)$ sends $D^*(X,H_X)$ into $D^* (Y,H_Y)$
    and induces therefore a morphism $f_*:=\Ad(W) \colon D^*(X,H_X) \to D^*
    (Y,H_Y)$. As for  
    $C^*$, one proves  
    that the induced map in K-theory does not depend on the choice of $W$. See
    again \cite[Lemma 3]{hry-mv}. 
    
     Up to tensoring
    with $\ell^2 (\NN)$, see \cite[Lemma 7.7]{hr-novikov}, it is always
    possible  
    to find an isometry $W$ satisfying the required
    two properties, which is the reason why we included this tensor product
    with $\ell^2(\naturals)$ in the definition of $D^*(X)$.
         
  By \cite[Lemma 7.8]{hr-novikov},  $K_* (D^* (X))$  is invariant under
  \emph{continuous} coarse homotopy.
  

\bigskip

To be able to use standard techniques from the K-theory of $C^*$-algebras,
given a subspace $Z\subset X$ we replace $C^* Z$ by an ideal $\relC{X}{Z}$ of
$C^*X$ as follows:

  \begin{definition}\label{def:rel_groups}
    Let $X$ be a proper metric space and $Z\subset X$ a closed subset. 
 Define
      $\relD{X}{Z}$ as the closure of those operators $T\in D_c^*(X)$ such
      that 
      there is an $R>0$ satisfying $\phi T=0=T\phi$ whenever $\phi\in C_c(X)$
      with $d(\supp(\phi),Z)>R$. 
Define
    $\relC{X}{Z}$\footnote{we deviate here from the notation employed by Roe, e.g.~$C^*_Z(X)$ for $\relC{X}{Z}$ in \cite[Definition
      3.10]{roe-cbms}. Our notations and definitions agree with those used in \cite{Siegel}.} as the closure of those $T\in C_c^*(X)$ such that
    \begin{enumerate}
    \item there is $R_T\ge 0$ satisfying $\phi T=0=T\phi$ whenever $\phi\in C_c(X)$ with
      $d(\supp(\phi),Z)>R_T$ and
    \item  $\forall\phi\in C_0(X\setminus Z)$ $\phi T$ and $T\phi$ are compact.
 \end{enumerate}
 Then
    $\relD{X}{Z}$ and $\relC{X}{Z}$  are ideals in $D^*(X)$.
  \end{definition}


We now describe equivariant versions of the constructions made so far.
  Assume therefore in addition that a discrete group $\Gamma$ acts freely and
  isometrically on the manifold $X$ and the Hermitean bundle $E$. It then also
  acts by unitaries on $H=L^2(E)$ and  we define 
  \begin{enumerate}
  \item $D^*(X)^\Gamma$ to be the norm closure of the $\Gamma$-invariant
    part $D^*_c(X)^\Gamma$, and its ideal $C^*(X)^\Gamma$ as the norm closure
    of $C^*_c(X)^\Gamma$. If $Z$ is a $\Gamma$-invariant subspace, we define
    in the corresponding way the ideals $\relD{X}{Z}^\Gamma$ and
    $\relC{X}{Z}^\Gamma$.
  \item The construction generalizes to an arbitrary proper metric space $X$
    with proper isometric $\Gamma$-action, using a
    $\Gamma$-adequate\footnote{adequate 
      requires a 
      little bit of extra care, compare \cite[Definition 5.13]{roe-cbms}:
      replacing $H$ by 
      $H\tensor l^2(\Gamma)\tensor l^2(\naturals)$ will do} $C_c(X)$-module
    $H$ with compatible unitary $\Gamma$-action.
  \item As indicated in the notation, one has suitable independence on $E$,
    along the way with the obvious generalization of functoriality to
    $\Gamma$-equivariant maps.
  \item 
  If 
   the quotient $V=X/\Gamma$ is a finite complex, then 
  $K_* (D^* (X)^\Gamma/C^* (X)^\Gamma)\simeq K_{*-1} (V)$;
  see
  \cite[Lemmas 5.14, 5.15]{roe-cbms}.
  \end{enumerate}

  \begin{lemma}\label{lemma:from-hry} (compare \cite[Lemma 1]{hry-mv},
    \cite[Proposition 3.8]{Siegel})\label{lem:subspace_Cstar} 
   Given a closed $\Gamma$-subspace $Z$ of a proper metric $\Gamma$-space $X$,
   the inclusion $Z\into 
   X$    induces K-theory isomorphisms
   $K_*(C^*(Z)^\Gamma)\xrightarrow{\iso} K_*(\relC{X}{Z}^\Gamma)$, $K_*(D^*
   (Z)^\Gamma)\xrightarrow{\iso} 
   K_*(\relD{X}{Z}^\Gamma)$. 
  \end{lemma}

Directly from the above results and  the short exact sequence
$$0\to C^*(X)^\Gamma \to D^*(X)^\Gamma\to D^* (X)^\Gamma/C^* (X)^\Gamma\to 0$$
we obtain the {\it Higson-Roe surgery sequence} for a $\Gamma$ manifold $X$ with quotient
$X/\Gamma$ compacts:
\begin{equation}\label{hr-surgery-sequence}
\cdots\to K_{n+1} (X/\Gamma) \to K_{n+1} ( C^* (X)^\Gamma) \to
     K_{n+1}((D^*(X)^\Gamma))  \to  K_{n} ( X/\Gamma) \to\cdots
\end{equation}

We will also be interested in a universal version of this sequence. First we
give a definition:

  \begin{definition}\label{def:universal_Gamma_objects}
    Let $\Gamma$ be a discrete group. Define
    \begin{equation*}
      K_*(C^*_\Gamma):=  \dirlim_{X\subset E\Gamma\text{ $\Gamma$-compact}}
      K_*(C^*(X)^\Gamma);\quad 
      K_*(D^*_\Gamma):= \dirlim_{X\subset E\Gamma\text{ $\Gamma$-compact}} K_*(D^*(X)^\Gamma).
    \end{equation*}
    Here, $E\Gamma$ is any contractible CW-complex with free cellular
    $\Gamma$-action, a  universal space for free $\Gamma$ actions.

  By coarse invariance of $C^*$ and \cite[Lemma 5.14]{roe-cbms}, there is a
  canonical isomorphism
  $K_*(C^*(X)^\Gamma)\iso C^*_{r}\Gamma$ for any free cocompact $\Gamma$-space
  $X$. Therefore the definition of $K_*(C^*_\Gamma)$ is along canonical
  isomorphisms and we get canonically $K_*(C^*_\Gamma)=K_*(C^*_r\Gamma)$. 
     Once this definition is given, we obtain immediately the (universal) Higson-Roe surgery sequence:
\begin{equation}\label{hr-surgery-universal}
\cdots\to K_{n+1} (B\Gamma) \to K_{n+1} ( C^*_\Gamma) \to
     K_{n+1}(D^*_\Gamma)  \to  K_{n} ( B\Gamma) \to\cdots
\end{equation}
which can be rewritten   as 
\begin{equation}\label{hr-surgery-universal-bis}
\cdots\to K_{n+1} (B\Gamma) \to K_{n+1} ( C^*_r \Gamma) \to
     K_{n+1}(D^*_\Gamma)  \to  K_{n} ( B\Gamma) \to\cdots
\end{equation}

It is proved  by Roe in \cite{roe-bc_coarse} that the homomorphism 
$K_{n+1} (B\Gamma) \to K_{n+1} ( C^*_r \Gamma)$ appearing in
\eqref{hr-surgery-universal-bis}
is precisely equal to the assembly map. This implies:
\begin{equation}\label{bc-remark-d}
\text{if $\Gamma$ is torsion free then the Baum-Connes conjecture for $\Gamma$
  is equivalent to }\;
K_{n+1}(D^*_\Gamma)=0 .
\end{equation}

    If $M$ is a proper complete metric space with a free cocompact isometric
    $\Gamma$-action, 
    there is a universal $\Gamma$-map $u\colon M\to E\Gamma$ with range in a
    $\Gamma$-finite subcomplex ($u$ is automatically coarse), and any two such
    maps are (coarsely continuously)
    $\Gamma$-homotopic. We therefore get canonical induced maps
    \begin{equation*}
      u_*\colon K_*(C^*(M)^\Gamma)\to K_*(C^*_\Gamma);\qquad u_*\colon
      K_*(D^*(M)^\Gamma)\to K_*(D^*_\Gamma).
    \end{equation*}
     Moreover, for $C^*(M)^\Gamma$ the map is a canonical isomorphism.

    More generally, if $W$ is a complete metric space with free
    $\Gamma$-action, $M\subset W$ is $\Gamma$-invariant and $M/\Gamma$ is
    compact then $\relD{W}{M}^\Gamma$ is the limit of $D^*(U_R(M))^\Gamma$ as
    $R\to\infty$, 
    where $U_R(M)$ is the closed $R$-neighborhood of $M$, again a
    $\Gamma$-compact metric space. We get a compatible system of universal
    maps to $E\Gamma$, all with image in finite subcomplexes, and an induced
    compatible system of maps in K-theory, giving rise to the maps
    \begin{equation*}
      K_*(D^*(M)^\Gamma) \xrightarrow[\iso]{i_*}  \lim_{R\to \infty} K_*(\relD{U_R(M)}{M}^\Gamma) \iso
      K_*(\relD{W}{M}^\Gamma)  \xrightarrow{u_*} K_*(D^*_\Gamma),      
    \end{equation*}
    whose composition is the universal map for $D^*(M)^\Gamma$.
  \end{definition}

\subsection{Index and $\rho$-classes}\label{subsect:ind-and-rho}

We now recall Roe's method of  applying $C^*$-techniques to the Dirac
operator to efficiently define primary and secondary  invariants for spin
manifolds in the context of coarse geometry.

Let $X$ be an arbitrary complete spin manifold with free isometric action by
$\Gamma$ of dimension $n>0$. Fix 
  an odd continuous chopping function
  $\chi\colon \reals\to\reals$, i.e.~$\chi(x)\xrightarrow{x\to +\infty} 1$. With
  the Dirac operator $D_X$ we
  now consider  $\chi(D_{X})$. Roe proves, using finite propagation speed of
  the wave operator and ellipticity, that this is an element in
  $D^*(X)^\Gamma$, compare \cite[Proposition 2.3]{roe-partitioning}.

  \begin{proposition}\label{prop:vanish}
    Assume that $Y\subset X$ is a $\Gamma$-invariant closed subset and the
    scalar curvature is uniformly positive outside 
    $Y$. Then 
   $\chi(D_{X})$ is an involution modulo 
   $\relC{X}{Y}^\Gamma$.

\noindent
   In particular, if we have uniformly positive scalar curvature, then
   $\chi(D_X)$ is an involution in $D^*(X)^\Gamma$.

\noindent
   For the other extreme, without any further curvature assumption,
   $\chi(D_X)$ is an involution modulo $C^*(X)^\Gamma$.
\end{proposition}
This important proposition is at the heart of the method. It is stated by Roe
\cite[Proposition 3.11]{roe-cbms} but without a full proof. 
A complete proof is given independently in Pape's thesis \cite[Theorem 1.4.28]{PapeThesis},
or by Roe in \cite[Lemma 2.3]{roe12:_posit}.

Recall that,
  given an involution $x$ in a $C^*$-algebra $A$, it defines in a canonical
  way the element $[\frac{1}{2}(x+1)]\in K_0 (A)$. If $n:=\dim(X)$ is odd, in
  the situation of Proposition \ref{prop:vanish} we 
  obtain the corresponding class $[D_X]:=[\frac{1}{2}(\chi(D_X)+1)] \in
  K_0(D^*(X)^\Gamma/\relC{X}{Y})$.

  If $n$ is even, we have to use the additional $\Gamma$-invariant grading of
  the spinor 
  bundle $L^2(S)=L^2(S_+)\oplus L^2(S_-)$. The operator $D_X$ and, because
  $\chi$ is an odd 
  function, $\chi(D_X)$ are odd with respect to this decomposition so that we
  obtain the positive part $\chi(D_X)_+\colon L^2(S_+)\to L^2(S_-)$. We choose
  a measurable bundle isometry $u\colon S_-\to S_+$ and obtain the induced
  isometry $U\colon L^2(S_-)\to L^2(S_+)$ covering $\id_X$ in the 
  $D^*$-sense\footnote{In \cite{roe-cbms}, it is only required that $U$ covers
    $\id_X$. However, as pointed out by Ulrich Bunke, to make sure that
    $U\chi(D_X)_+\in D^*(X)^\Gamma$ one needs the stronger assumption.

  In a previous version of the paper (the version published in Journal of
  Topology) we were only requiring that $U$ covers $\id_X$ in the
  $D^*$-sense. But only if we are as specific as in the new version, the
  K-theory class we introduce is indeed well-defined. This follows because now
  the possible 
  choices form a contractible set, and by homotopy invariance of K-theory
  everything is well defined. For more details on this
issue compare \cite[Section 2.2]{PiazzaSchick_surgery}}. 

  Then $U\chi(D_X)_+$ is a unitary in $D^*(X)^\Gamma/\relC{X}{Y}^\Gamma$ and
represents $[D_X] \in K_1(D^*(X)^\Gamma/\relC{X}{Y}^\Gamma)$.

  \begin{definition}\label{def:coarse_ind}
    Let $(X,g)$ be a complete Riemannian spin manifold of dimension $n>0$ with
    isometric free action of $\Gamma$. Define
    \begin{equation*}
      \Ind^{{\rm coarse}} (D_X) := \partial([D_X]) \in K_n(C^*(X)^\Gamma).
    \end{equation*}
  Here $\partial$ is the boundary map of the long exact sequence of the
  extension $0\to C^*(X)^\Gamma\to D^*(X)^\Gamma\to
  D^*(X)^\Gamma/C^*(X)^\Gamma\to 0$.

     Observe that, if we have uniformly positive scalar curvature outside of
    $Y$, we have a canonical lift to
    \begin{equation*}
      \Ind^{{\rm rel}}(D_X) := \partial([D_X]) \in 
      K_n(\relC{X}{Y}^\Gamma).
    \end{equation*}
 If we have uniformly positive scalar curvature throughout, we
   define a secondary invariant, the $\rho$-class of the metric $g$, as 
   \begin{equation}\label{def-rho}
     \rho(g):= [D_X] \in K_{n+1}(D^*(X)^\Gamma).
   \end{equation}
    Finally, if $X/\Gamma$ is compact, there is the canonical map to
   $K_{n+1}(D^*_\Gamma)$ of Definition \ref{def:universal_Gamma_objects} and
   we define  $\rho_\Gamma(g)\in
   K_{n+1}(D^*_\Gamma)$, the $\rho_\Gamma$-class of $g$,
   as the image of $\rho(g)$ under this map:
   \begin{equation}\label{def-rho-gamma}
   \rho_{\Gamma} (g):= u_* (\rho (g)) \;\;\in\;\;K_{n+1}(D^*_\Gamma).
   \end{equation}
   \end{definition}
   
   \begin{remark}\label{remark:torsion}
   It is important to point out that in contrast to the $\rho$-class 
   $\rho (g)\in
   K_{n+1}(D^*(X)^\Gamma)$, the $\rho_\Gamma$-class  
   $\rho_{\Gamma} (g)\in K_{n+1}(D^*_\Gamma)$ vanishes for groups
   without torsion, at least for those for which the Baum-Connes conjecture holds.
   See the fundamental remark appearing in \eqref{bc-remark-d}. This means we
   expect  $\rho_{\Gamma} (g)$
   to be different from zero only for groups
   $\Gamma$ {\it with} torsion.\\
  Basic non-trivial examples of $\rho_\Gamma(g)$ for $\Gamma$ with torsion are
  considered in \cite{higson-roe4}. 
      \end{remark}

Notice that the $\rho$-class is well defined whenever the Dirac operator $D_X$ is $L^2$-invertible;
we denote it $\rho(D_X)$ in this more general case. In fact, we will sometime employ this notation also for the spin Dirac operator
associated to a positive scalar curvature metric.

\subsection{ Delocalized APS-index theorem} \label{sec:index_class_boundary}

\begin{set-up}\label{geometric-set-up}
Let now $(W,g_W)$ be a $n$-dimensional Riemannian
spin manifold with boundary, complete as metric
space\footnote{i.e.~every Cauchy sequence converges}. We
denote its  
boundary $(M,g_M)$, and we assume always that we have product structures near
the boundary. We assume 
that the scalar curvature of $g_M$ is uniformly positive, and that $\Gamma$
acts freely, isometrically and cocompactly on $W$ and therefore also on $M$. We denote
the quotient  of $(W,g_W)$ by the action of $\Gamma$  as $(Y,g_Y)$, a compact Riemannian manifold with boundary.
Associated to these data is $W_\infty=W\cup_M M\times [0,\infty)$ with extended
product structure on the 
cylinder. This defines a complete Riemannian metric $g$ on  $W_\infty$
and we then have uniformly positive scalar curvature outside $W\subset
W_\infty$.
\end{set-up}

The considerations of the previous subsection apply now to the pair $(W\subset W_\infty)$ and we obtain
therefore a class $\Ind^{{\rm rel}}(D_{W_\infty}) \in K_n(\relC{W_\infty}{W}^\Gamma)$
and thus a class
\begin{equation}\label{coarse-class}
  \coarseind(D_W):=c_*^{-1}\Ind^{{\rm rel}}(D_{W_\infty}) \in K_n(C^*(W)^\Gamma).
\end{equation}

Here we use the canonical inclusion $c\colon C^*(W)^\Gamma\to
\relC{W_\infty}{W}^\Gamma$ which induces an isomorphism in K-theory by
Lemma~\ref{lemma:from-hry} below. 

Let us remark here that, under the canonical isomorphism
$K_n(C^*(W)^\Gamma)\iso K_n(C^*_{r}\Gamma)$,  this index class corresponds
to any of the other APS-indices for manifolds with boundary defined in this
context, e.g.~using the Mishchenko-Fomenko approach and the b-calculus or
using APS-boundary conditions, compare Section \ref{sec:index_technicalities}.

The passage from $C^*X$ to $D^*X$ corresponds to the passage to the
delocalized part of the index information (we will explain this later). This delocalized part we can
compute by a K-theoretic version of the APS-index theorem:

\begin{theorem}[Delocalized APS-index theorem] \label{theo:k-theory-deloc}  
Let $(W,g_W)$ be an even dimensional Riemannian spin-manifold with boundary
$\boundary W$
such that $g_{\boundary W}$ 
has positive scalar curvature. Assume that $\Gamma$ acts freely isometrically
and $W/\Gamma$ is compact. Then
\begin{equation}\label{k-theory-deloc}
\iota_* (\coarseind (D_W))= j_*(\rho(g_{\boundary W})) \quad\text{in}\quad K_{0}
(D^*(W))^\Gamma).
\end{equation}
Here, we use $j\colon D^*(\boundary W)^\Gamma\to D^*(W)^\Gamma$ induced by the
inclusion 
$\boundary W\to W$ and $\iota\colon C^*(W)^\Gamma\to D^*(W)^\Gamma$ the inclusion.
\end{theorem}

\begin{corollary}\label{corollary-main}\label{gamma-index}
  By functoriality, using the canonical $\Gamma$-map $u\colon W\to E\Gamma$ of
  Definition \ref{def:universal_Gamma_objects},
  we have $   \iota_* u_*(\coarseind (D_W))  =\rho_\Gamma(g_{\boundary W})
     \quad\text{in}\quad K_0 (D^*_\Gamma)$. 
 If we define 
$ \coarseind_\Gamma (D_W):=u_* (\coarseind (D_W))$ 
in
 $K_0(C^*_\Gamma)$, 
 then the last equation reads
  \begin{equation}\label{main-index-eq}
   \iota_* (\coarseind_\Gamma (D_W))  =\rho_\Gamma(g_{\boundary W})
   \quad\text{in}\quad K_0(D^*_\Gamma).
 \end{equation}
\end{corollary}
 
This gives immediately bordism invariance of the $\rho$-classes:
\begin{corollary}\label{corol:part_index_special}
  Let $(M_1,g_1)$ and $(M_2,g_2)$ be two odd-dimensional free cocompact spin
  $\Gamma$-manifolds of positive scalar curvature. Assume that they are
  bordant as manifolds with positive scalar curvature, i.e.~that there is a
  Riemannian spin manifold $(W,g)$ with free cocompact $\Gamma$-action such
  that $\boundary W=M_1\disjointunion -M_2$, $g$ has positive
  scalar curvature and restricts to $g_j$ on $M_j$. Then
  \begin{equation*}
    \rho_\Gamma(g_1) = \rho_\Gamma(g_2) \in K_0 (D^*_\Gamma).
  \end{equation*}
\end{corollary}
\begin{proof}
  The rho-class is additive for disjoint union and changes sign if one
  reverses the spin structure. Because $W$ and $W_\infty$ have uniformly
  positive scalar curvature, $\Ind(D_W)=0$; thus  $\Ind_\Gamma (D_W)=0$.
  The assertion now follows
  directly from
  Corollary \ref{corollary-main}.%
\end{proof}

\begin{remark}\label{remark:bordism-invariance}
Notice that bordism invariance holds only for $\rho_\Gamma$-classes; indeed,
we need a common $K$-theory group where we can compare the two invariants.
Precisely because of this last observation,
the following variant of Corollary \ref{corol:part_index_special}
holds:\\
 Let $(M_1,g_1)$ and $(M_2,g_2)$ be two free cocompact spin
  $\Gamma$-manifolds of positive scalar curvature endowed with 
 $\Gamma$-equivariant  reference maps $f_1, f_2$ to a Hausdorff topological 
 $\Gamma$-space $\widetilde{X}$ with compact quotient $X:= \widetilde{X}/\Gamma$. 
 Assume that  there exists a
  Riemannian spin manifold $(W,g)$ as in  Corollary \ref{corol:part_index_special}
  endowed with a $\Gamma$-equivariant reference map $F\colon W\to \widetilde{X}$ such
  that $F|_{M_j}=f_j$. Then, defining $\rho_{\widetilde{X}} (g_j) := (f_j)_* \rho (g_j)\in 
  K_0 (D^* (\widetilde{X})^\Gamma)$, we have the following identity:
  \begin{equation}\label{bordism_X}
    \rho_{\widetilde{X}}  (g_1) = \rho_{\widetilde{X}} (g_2) \in K_0 ( D^* (\widetilde{X})^\Gamma ).
  \end{equation}
  
 \begin{proof}
 Denote by $\iota_{\widetilde{X}}: C^* (\widetilde{X})^\Gamma \to D^* (\widetilde{X})^\Gamma$ the inclusion
 and similarly for $\iota_W$. Let $j_1$ and $j_2$ be the natural inclusions
 $M_j\hookrightarrow W$. Then, from Theorem \ref{theo:k-theory-deloc}  
 we get
 $$(\iota_W)_* (\Ind (D_W))= (j_1)_*   \rho (g_1) - (j_2)_*   \rho (g_2)\quad\text{in}\quad K_0 ( D^* (W)^\Gamma )$$
 We now apply $F_* \colon K_0 ( D^* (W)^\Gamma )\to K_0 ( D^* (\widetilde{X})^\Gamma )$. Since $F\circ j_1=f_1$
 and $F\circ j_2 = f_2$ and since $F_* (\iota_W)_*=(\iota_{\widetilde{X}})_* F_*$, with the $F_*$
 on the right hand side going from $K_0 ( C^* (W)^\Gamma )$
 to
$  K_0 ( C^* (\widetilde{X})^\Gamma )$,  we see that
 $$ (\iota_X)_* (F_*  (\Ind (D_W))= (f_1)_* \rho (g_1) - (f_2)_*\rho (g_2)$$
 Since the left hand side vanishes (recall that $g$ on $W$ is of positive scalar
 curvature) this is precisely what we wanted to prove.
   \end{proof}
  \end{remark}

\begin{remark}\label{rem:odd_rem}
  We are convinced that the theorem also is correct if $\dim(W)$ is odd. In the
  present paper we only deal with the even case, By using $Cl_n$-linear Dirac
  operators and an appropriate setup for $Cl_n$-linear (also called
  $n$-multigraded) cycles for K-theory, we expect that our method should
  generalize to all dimensions and also to the refined invariants in real
  K-theory one can get that way. We plan to address the details in future
  work.
\end{remark}

\begin{remark}
As we shall see, Theorem \ref{theo:k-theory-deloc} has a surprisingly intricate proof.
A different approach for proving it would be to develop a theory for the
Calderon projector $P$
associated to a Dirac-type operator 
on a Galois covering with boundary. 
In this direction, recall
the classical formula for the APS numeric index in terms of the Calderon projection $P$
and the APS projection $\Pi_{\geq}$: $\ind^{{\rm APS}} D^+={\bf i} (\Pi_{\geq},P)$.
 If one were able to extend this formula to the APS-index class, then the theorem
 would follow provided one could establish, in addition, that
the image of  the class of the Calderon projector $[P]$ in $K_0 (D^* (W)^\Gamma)$ vanishes. It
 would be very interesting to work out this alternative approach to Theorem
 \ref{theo:k-theory-deloc}, which seems to be, however, quite an intricate
 question. A first step in this direction is carried out in \cite{Antonini},
 where the Calderon projector for $C^*$-module coefficients is constructed.
\end{remark}

\begin{example}
  The morphisms $C^*(M)^\Gamma\to C^*(W)^\Gamma\to \relC{W_\infty}{W}^\Gamma$
  induce (canonical) 
  isomorphisms in K-theory by Lemma \ref{lemma:from-hry} and because $M\to W$
  is a coarse equivalence, as $W/\Gamma$ is compact. Consequently we can map
  $\coarseind(D_W)$ also to 
  $K_*(D^*(M)^\Gamma)$ and compare its image there to
  $\rho(g_M)$. 

  It turns out that in general these two objects are different, so that a
  corresponding sharpening of Theorem \ref{theo:k-theory-deloc} is not
  possible. Indeed, an additional secondary term, a rho-class of a bordism,
  shows up. This secondary class appears naturally when one gives a proof of
  bordism invariance of the rho-index using suitable exact sequences of
  K-theory of Roe algebras and the principle that ``boundary of Dirac is
  Dirac''. We plan to work this out in a sequel publication.
  
  Explicitly, take $W=D^{n+1}$ with $\boundary W=S^n$, with the standard
  metrics (slightly modified to have product structure near the boundary, but
  clearly with positive scalar curvature as long as $n>1$).

  Because of overall positive scalar curvature, $\coarseind(D_W)\in
  K_*(C^*(W))=K_*(\complexs)$ vanishes, and so does its image in
  $K_*(D^*(S^n))$.

  On the other hand, the Dirac operator on $S^n$ represents the fundamental
  class, a non-trivial element in $K_n (S^n)$. By the
  commutativity of the diagram \eqref{eq:StolzToAna}, another main
  theorem of this paper, $\rho(g_{S^n})\in
  K_n(D^*(S^n))$ has to be non-trivial, being mapped to a non-trivial element in
  $K_n(D^*(S^n)/C^*(S^n))=K_n(S^n)$. Observe that this is a purely
  topological phenomenon, having nothing to do with analysis. 
\end{example}

\subsection{Secondary index theorem for $\rho$-classes on partitioned
  manifolds}\label{subsect:intro-partitioned}

In this section, we formulate a partitioned manifold secondary index theorem,
for the $\rho$-class on a manifold of uniformly positive scalar curvature. 

For this aim, let $W$ be a (non-compact) Riemannian spin manifold of
dimension $n+1$ with
isometric free $\Gamma$-action and assume that there is a $\Gamma$-invariant
two-sided hypersurface $M\subset W$ such that $M/\Gamma$ is compact. We get a
decomposition $W=W_-\cup_M W_+$.

Let us quickly recall the primary partitioned manifold index theorem. The
classical case is $\Gamma=\{1\}$, then we obtain $\Ind(D_W)\in
K_{n+1}(C^*(W))$. The partition allows to construct a map (showing up in a
corresponding Mayer-Vietoris sequence as in Section
\ref{subsect:mayer-vietoris}) to $K_n (C^*(M))=K_n(\complexs)$. The 
partitioned manifold theorem of Roe \cite{roe-partitioning} then simply states that
the image of $\Ind(D_W)$ 
under this map is  $\ind(D_M)$. The corresponding statement for non-trivial
$\Gamma$ and even $n$ is covered in \cite{Zadeh}. 

We now treat the same question for the secondary rho class of manifolds with
uniformly positive scalar curvature. Indeed, let us first give a direct
definition of the partitioned manifold rho-class, similar to the definition of
the partitioned manifold index as given by Higson \cite{higson-cobordism}.

\begin{definition}
  Assume, in the above situation, that $W$ has dimension $n+1$
  and uniformly positive scalar
  curvature. Then 
  we constructed
  $\rho(D_W)\in K_{n+2}(D^*(W)^\Gamma )$. Consider the image of $[D_W]$ under the
  $D^*$-Mayer-Vietoris   boundary map for the decomposition of $W$ into $W_+$
  and   $W_-$ along $W$ (discussed in Section
  \ref{subsect:mayer-vietoris}):
  $\delta_{{\rm MV}} [D_W]
  \in K_{n+1} (D^*({M})^\Gamma)$.
   We set
  \begin{equation}\label{pre-pm-rho}
   \rho^{{\rm pm}}(g):= \delta_{{\rm MV}} [D_W]
  \in K_{n+1} (D^*({M})^\Gamma)
  \end{equation}
  and we call it the {\it partitioned manifold $\rho$-class}
  associated to the partitioned manifold $W=W_-\cup_M W_+$.
We shall be mainly concerned with a universal version of 
this class: we consider the canonical map $u\colon M\to E\Gamma$
 and we set
\begin{equation}\label{pm-rho}
  \rho_\Gamma^{{\rm pm}}(g):= u_* (\delta_{{\rm MV}} [D_W])\in K_{n+1}(D^*_\Gamma).
\end{equation}
We call this secondary invariant the {\it partitioned manifold
  $\rho_\Gamma$-class} 
associated to $W=W_-\cup_M W_+$.
\end{definition}


\begin{theorem} \label{theo:part_mf_abst}
Let $(W,g)$ be a connected spin manifold partitioned by a
hypersurface 
$M$ into $W_-\cup_M W_+$. Let $\Gamma$ act freely on $(W,M)$. Let $\dim(W)=n+1$ be even. Assume that the metric $g$
on $W$ has uniformly positive scalar 
curvature and that the metric on a tubular neighborhood of the hypersurface
$M$ has product structure, so that the induced metric $g_M$ also has positive
scalar curvature. Assume, finally, that $M/\Gamma$ is compact.
Then
\begin{equation*}
  \rho^{{\rm pm}}_\Gamma (g) =  \rho_{\Gamma} (g_M)\in K_{n+1} (D^*_\Gamma).
\end{equation*}
\end{theorem}


\begin{remark}
  We are convinced that the assertion of Theorem \ref{theo:part_mf_abst} also
  holds if $\dim(W)$ is odd. As detailed in Remark \ref{rem:odd_rem}, with
  appropriate new multigrading input our method might carry over. Again, we
  hope to work this out in the future.
\end{remark}

\begin{corollary}
  Let $W$ be as in Theorem \ref{theo:part_mf_abst} with two
  $\Gamma$-equivariant metrics $g^0,g^1$
  of uniformly positive scalar curvature (and in the same coarse equivalence
  class) which are of product type near $M$. If  $g^0,g^1$ are connected by a
  path of
  uniformly positive $\Gamma$-equivariant metrics $g^t$ in the same coarse
  metric class (not necessarily
  product near $M$) then $\rho^{{\rm pm}}_\Gamma (g^0_M)=\rho^{{\rm
      pm}}_\Gamma (g^1_M) \in K_{n+1}(D^*_\Gamma)$.
  \end{corollary}
\begin{proof}
  We simply have to observe that we get a
  homotopy $\rho^{{\rm pm}}_\Gamma (g^t)$ between $\rho^{{\rm pm}}_\Gamma
  (g^0)$ and $\rho^{{\rm pm}}_\Gamma (g^1)$ in $D^*_\Gamma$ and then apply
  homotopy invariance of K-theory.
\end{proof}

 As an application of this Corollary,  assume that $M/\Gamma$, which
 is assumed to be compact, has two metrics
  $g_0$, $g_1$ with $\Gamma$-invariant lifts
 $ \tilde g_0$, $\tilde g_1$ that have the property that
  $\rho_\Gamma (\tilde g_0)\ne \rho_\Gamma (\tilde g_1)\in K_{n+1}(D^*_\Gamma)$.
Of course, this implies that the two metrics are not concordant on $M/\Gamma$.
Stabilize by taking the product with $\reals$ (with the standard metric).
  We can now conclude that even with the extra room on $M/\Gamma\times \reals$
  we can 
  not deform $g_0+dt^2$ to $g_1+dt^2$ through metrics of uniformly positive
  scalar curvature. This follows directly from the Corollary.

%




\medskip
  The strategy of proof for this $\rho$-version of the partitioned manifold
  index theorem is the same as the classical one:

  \begin{enumerate}
  \item we prove it with an explicit calculation for the product case;
   \item we prove that the partitioned manifold rho-class depends only on a
     small neighborhood of the hypersurface.
  \end{enumerate}

  \begin{remark}
    In the situation of Theorem \ref{theo:part_mf_abst}, both $\rho^{\rm pm}(g)$
    and $\rho(g_M)$ are defined in $K_*(D^*(M)^\Gamma)$. However, our method
    does not give any information about equality of these classes, only about
    their images in $K_*(D^*_\Gamma)$. This is in contrast to Theorem
    \ref{theo:k-theory-deloc}, where the equality is established in
    $K_*(D^*(W)^\Gamma)$.

    On the other hand, we also don't have an example where the partitioned
    manifold $\rho$-class does not coincide with the $\rho$-class of the cross
    section. It is an interesting challenge to either find such examples, or
    to improve the partitioned manifold secondary index theorem. The latter
    would be important in particular
    in light of applications like the stabilization problem we just discussed:
    if $g_1,g_2$ on $M$ are positive scalar curvature metrics which are not
    concordant, is the same true for $g_1+dt^2$ and $g_2+dt^2$ on
    $\reals\times M$?
  \end{remark}

\subsection{Mapping the positive scalar curvature sequence to analysis} 

  The Stolz exact sequence is the companion for positive scalar curvature of
  the surgery exact sequence in the classification of high dimensional
  manifolds. The latter
  connects the structure set, consisting of all manifold structures in a given
  homotopy type, with the generalized homology theory given by the L-theory
  spectrum and the algebraic L-groups of the fundamental group. 

  Similarly, Stolz' sequence connects the ``structure set''
  $\Pos^{\spin}$ (compare Definition \ref{def:Stolz_seq_terms}), which contains the
  equivalence classes of metrics of positive scalar curvature to the generalized homology group $\Omega^{\spin}$ and to
  $R^{\spin}(X)$. The latter indeed is a group which only depends
  on the fundamental group of $X$. It is similar to the geometric definition
  of L-groups. Missing until now is an algebraic and computable description of
  these $R$-groups, in contrast to the $L$-groups of surgery. 

In this subsection, we construct a map from
the Stolz  positive scalar curvature exact sequence to analysis. We give a
picture which 
describes the transformation as directly as possible, using 
indices defined via coarse geometry.

\begin{definition}\label{def:Stolz_seq_terms}
  Given a reference space $X$ (often
  $X=B\Gamma$), define
  \begin{itemize}
  \item $\Pos^{\spin}_n(X)$ as the set of singular bordism classes $(M,f\colon
    M\to 
    X,g)$  of  $n$-dimensional closed 
    spin manifolds $M$ together with a reference map $f$ and a positive scalar
    curvature metric $g$ on $M$. A bordism between  $(M,f\colon M\to X,g)$
    and  $(M^\prime,f^\prime\colon M^\prime\to X,g^\prime)$
    consists of a compact manifold with boundary $W$, with $\pa W= M\sqcup
    (-M^\prime)$, a reference map $F\colon W\to X$ restricting
    to $f$ and $f^\prime$ on the boundary
     and a positive scalar curvature metric on $W$
    which has product structure near the boundary and  restricts to $g$ and $g^\prime$
    on the boundary.
  \item $R^{\spin}_{n+1}(X)$ as the set of bordism classes $(W,f,g)$ where $W$ is a
    compact $(n+1)$-dimensional spin-manifold, possibly with boundary, with a
    reference map 
    $f\colon W\to X$, and with a positive scalar curvature metric \emph{on the
      boundary} when the latter is non-empty. 
     Two triples 
$(W,f,g_{\pa W})$, $(W^{\prime},f^{\prime},g^{\prime}_{\pa
W})$ are bordant 
if  there is a bordism with positive scalar curvature
between the two boundaries, call it $N$, such that 
$$Y:=W\cup_{\pa W} N\cup_{-\pa W^{\prime}} (-W^\prime)$$ 
is the boundary of a spin manifold $Z$. The reference maps to $X$ have to
  extend over $Z$.
By the surgery method for the construction of positive scalar curvature
    metrics, this set actually depends only on the fundamental group of $X$
    if $X$ is connected, compare \cite[Section 5]{Rosenberg-Stolz_survey}. 
  \item $\Omega_n^{spin}(X)$ is the usual singular spin bordism group of $X$.
  \end{itemize}
\end{definition}

\begin{proposition}  As a direct consequence of the definitions we get a long
  exact sequence, the \emph{Stolz exact sequence}
  \begin{equation*}
    \to \Pos^{\spin}_n(X)\to \Omega_n^{spin}(X)\to R^{\spin}_{n}(X) \to
    \Pos^{\spin}_{n-1}(X)\to
  \end{equation*}
  with the obvious boundary or forgetful maps.  
\end{proposition}

\begin{theorem}\label{theo:commute-r-pos}
For $X$ a compact space with fundamental group $\Gamma$ and universal covering
$\tilde X$, there exists a well defined and commutative diagram, if $n$ is odd,
\begin{equation}\label{eq:StolzToAnaX}
  \begin{CD}
   @>>> \Omega^{\spin}_{n+1} (X) @>>>   R^{\spin}_{n+1}(\Gamma) @>>>
    \Pos^{\spin}_n (X) @>>> \Omega^{\spin}_n (X) @>>>\\ 
    && @VV{\beta}V @VV{\Ind_\Gamma}V  @VV{\rho_\Gamma}V @VV{\beta}V\\
     @>>>  K_{n+1} (X) @>>> K_{n+1} ( C^*_r\Gamma) @>>>
     K_{n+1}(D^*(\tilde X)^\Gamma)  @>>>  K_{n} (X) @>>>\\
    \end{CD}
\end{equation}
We also get a universal commutative diagram
\begin{equation}\label{eq:StolzToAna}
  \begin{CD}
   @>>> \Omega^{\spin}_{n+1} (B\Gamma) @>>>   R^{\spin}_{n+1}(B\Gamma) @>>>
    \Pos^{\spin}_n (B\Gamma) @>>> \Omega^{\spin}_n (B\Gamma) @>>>\\ 
    && @VV{\beta}V @VV{\Ind_\Gamma}V  @VV{\rho_\Gamma}V @VV{\beta}V\\
     @>>>  K_{n+1} (B\Gamma) @>>> K_{n+1} ( C^*_r\Gamma) @>>>
     K_{n+1}(D^*_\Gamma)  @>>>  K_{n} ( B\Gamma) @>>>\\
    \end{CD}
\end{equation}
\end{theorem}

\begin{remark}
  As soon as the extension to arbitrary dimensions of our secondary index
  theorem   \ref{theo:k-theory-deloc} has been carried out (as indicated in
  Remark \ref{rem:odd_rem}), also Theorem \ref{theo:commute-r-pos} extends to
  the arbitrary dimensions.
\end{remark}

\begin{remark}
  In their seminal papers \cite{higson-roeI,higson-roeII,higson-roeIII}
  Higson and Roe carry out a  a program similar to the one developed here:
  they construct a map from 
  the surgery exact sequence in topology to exactly the same K-theory exact
  sequence 
  showing up in \eqref{eq:StolzToAna} (with $2$ inverted). Their construction
  is not quite as
  analytic as ours: it is not based on the index of the signature operator but
  rather on the manipulation of Poincar\'e duality complexes. In
  \cite{PiazzaSchick_surgery} we develop a direct analogue of our index
  theoretic 
  construction for the surgery exact sequence in topology. Notice that the
  analysis is more 
  difficult than the one developed here, given that the signature operator
  attached to an element of the structure set will not be
   invertible; similarly, the boundary-signature operator attached to an
   element of the
   L-groups of the fundamental group will not be invertible. One can use
  the homotopy equivalences built in the definition of the L-groups and  the structure set
   in order to obtain a smoothing perturbation
  which makes the signature operator invertible, as in \cite{PS1}. The issue
  is then to extend the constructions of index and rho-classes and the proofs
  of the secondary index theorems to this more general class of Dirac-type
  operators with smoothing perturbation (making the sum invertible).

  We have completed this program in \cite{PiazzaSchick_surgery},
  reproving the main results of
  \cite{higson-roeI,higson-roeII,higson-roeIII} with purely operator theoretic
  methods. The corresponding general index theorems should be useful in other
  contexts, as well.
\end{remark}

\begin{remark}
  Beyond the extension of the method to the surgery exact sequence, a second
  goal for future work is to continue and map further from the K-theory exact
  sequence of \eqref{eq:StolzToAna} to a suitable exact sequence in cyclic
  (co)homology which should then allow to obtain systematically numerical
  higher invariants. To achieve this, one has to overcome further analytic
  difficulties as the algebra $D^*X$ is  too large to allow for easy
  constructions of (higher) traces on it. Higson and Roe \cite{higson-roe4}
  carry out a small part of this program. They check that the pairing with
  the trace coming from a virtual representation of dimension zero (which
  gives rise to the Atiyah-Patodi-Singer rho invariant) is compatible with the
  K-theory exact sequence. It turns out that their construction of the
  relevant map on $K_*(D^*X)$ is very delicate. In \cite{Wahl_higher_rho}
  Wahl
  extends this to some more refined invariants, but working directly with the
  surgery exact sequences and its specific properties and, more importantly,
   mapping directly
  to cyclic homology, or rather non-commutative de-Rham homology, 
  as it was done in \cite{LPPSC}
  for the Stolz sequence.
\end{remark}

We now describe the structure of the rest of the paper. In Section
\ref{sec:index_technicalities} we review several alternative and previously
used definitions of higher indices, in particular for manifolds with boundary,
and check that they coincide with the approach via coarse $C^*$-algebras
which we 
have described above (in the contexts where this makes sense). This puts
``coarse index theory'' in the context of usual index theory and allows us to
use a few known properties of indices (like bordism invariance and
gluing formulas) without having to prove them again in the coarse
setting. In Section \ref{sec:K-theory-homs} we will work out basic properties
of the K-theory of coarse $C^*$-algebras which we use in the course of the
proofs. Section \ref{sec:proofs} finally is devoted to the proofs of the main
Theorems \ref{theo:part_mf_abst} and
\ref{theo:commute-r-pos}, implementing the program set out above.

\begin{remark}
  After the first publication of the present paper in the arXiv,
   Xie and Yu, in the preprint \cite{Xie-Yu},
  treated the problem with a different method. They use Yu's localization
  algebras and an exterior product structure between K-homology and the
  analytic structure group to reduce the proof of the main result of this paper
  to the known
  behavior of the K-homology fundamental class under the Mayer-Vietoris
  boundary map. Where for us the main difficulty lies in the explicit index
  calculation in the model situation, for them the main difficulty is the
  explicit calculation of certain exterior products which uses the full force
  of KK-theory.
 Their method does cover even and odd dimensions at the same time.

  Moreover, also Paul Siegel announced a proof of the general case, along
  similar lines as Xie and Yu. In his Ph.D. thesis, he develops a new model for
  K-homology and the structure set and develops an exterior product between
  those, and calculates the exterior product between a rho-class and a
  fundamental class. Paul Siegel has announced that he  proved
 the compatibility
  between exterior product and Mayer-Vietoris. Again, this would lead to a
  proof of our main theorems in all dimensions, and would generalize to real
  K-theory. 
\end{remark}

\medskip
\noindent
{\bf Acknowledgments.}  We thank John Roe for useful discussions in an
early stage of this project and also for very useful comments on several
versions of the manuscript.  Thomas Schick thanks  Sapienza Universit\`a 
di Roma for a three-months guest professorship during which the most important
steps of this project could be carried out. Paolo Piazza thanks the
Mathematische Institut and the Courant Center ``Higher order structures in
mathematics'' for their hospitality for several week-long visits to
G\"ottingen.

\section{Coarse, $b$- and APS index classes}
\label{sec:index_technicalities}

The goal of this section is to give alternative descriptions of the relative coarse index class
  $\Ind^{{\rm rel}}(D_{W_\infty}) \in K_{n} (\relC{W_\infty}{W}^\Gamma)$, $n=\dim W$,
connecting it with classes that have already been defined in the literature. We also explain
why we look at our index theorem as a delocalized APS index theorem.
\subsection{Index classes in the closed case}
First of all we  tackle the
analogous problem in the boundaryless case. Thus, let $V$ be a complete
spin\footnote{our arguments actually apply to any Dirac-type operator} Riemannian 
manifold with a free, isometric, cocompact spin structure preserving action of
$\Gamma$. We denote the quotient $V/\Gamma$,
a compact spin manifold without boundary,
by $Z$; we thus get a Galois $\Gamma$-covering $ V\xrightarrow{\pi} Z$. 
We denote the spinor bundles on $V$ and $Z$ by $S_V$ and $S_Z$ respectively; $S_V$ is $\Gamma$-equivariant and $S_Z$ is obtained from $S_V$ by passing to the quotient.
There are five $C^*$-algebras which we consider:
\begin{itemize}
\item $C^* (V)^\Gamma$, the Roe algebra we have defined  in subsection \ref{subsect:ind-and-rho};
\item the $C^*$-algebra $C^* (G)$ defined by the groupoid $G$ associated to the 
$\Gamma$-covering $\Gamma -  V\to Z$; this is the groupoid with set of arrows $V\times_\Gamma V$,
units $Z$ and source and range maps defined by $s[v,v^\prime]=\pi (v^\prime)$ and
$r[v,v^\prime]=\pi (v)$. We also consider the Morita equivalent $C^*$-algebra
$C^* (G,S_V)$,
which is defined by taking the closure of the algebra of smooth integral
kernels $C^\infty_c (G, (s^* S_Z)^*\otimes r^* S_Z)$;
\item the $C^*$-algebra of compact operators $\KK (\mathcal{E})$ of the $C^*_r \Gamma$-Hilbert
module $\mathcal{E}$ defined by taking the closure of the  pre-Hilbert $\CC\Gamma$-module 
$C^\infty_c (V,S_V)$;
\item the $C^*$-algebra of compact operators $\KK (\mathcal{E}_{{\rm MF}})$ of the Mishchenko-Fomenko
$C^*_r \Gamma$-Hilbert
module $\mathcal{E}_{{\rm MF}}$ which is, by definition, $L^2 (Z, S_Z\otimes \mathcal{V}_{{\rm MF}})$; here
 $\mathcal{V}_{{\rm MF}}$ denotes the Mishchenko bundle: $\mathcal{V}_{{\rm MF}}
 := V\times_\Gamma C^*_r \Gamma$. We can also consider $\mathcal{E}^1_{{\rm MF}}:= H^1 (Z, S_Z\otimes \mathcal{V}_{{\rm MF}})$, the first Sobolev $C^*_r \Gamma$-module. 
 \item the reduced $C^*$-algebra of the group $\Gamma$: $C^*_r \Gamma$.
 \end{itemize}
 The relationships between these $C^*$-algebras are as follows:
 \begin{equation}\label{c-star-iso}
 C^* (V)^\Gamma=C^* (G,S)\simeq \KK (\mathcal{E})\simeq \KK (\mathcal{E}_{{\rm MF}})\simeq \K \otimes 
 C^*_r \Gamma.
 \end{equation}
 where all the isomorphisms are canonical.
 The first equality, with $L^2 (V, S_V)$ chosen as Hilbert $C_0 (V)$-module for $ C^* (V)^\Gamma$,
  follows from the inclusion $$C^\infty_c (G, (s^* S_Z)^*\otimes r^* S_Z)\subset C_c (V)^\Gamma\,.$$
  The second isomorphism is a special case of
  the corresponding result for foliated bundles in \cite{mn}, and was
  well known before. The third
  isomorphism is induced by a canonical isomorphism of $\CC\Gamma$-modules,
$\psi\colon C^\infty_c (V,S_V)\to C^\infty_c (V,S_V\otimes \CC\Gamma)^\Gamma\equiv C^\infty (Z,S_Z\otimes 
\mathcal{V}_{{\rm alg}})$ with $\mathcal{V}_{{\rm alg}}
 := V\times_\Gamma \CC\Gamma$, see \cite[Proposition 5]{LottI} (it will
 suffice to replace $\mathcal{B}^\omega$ there with $\CC\Gamma$).
 The last isomorphism is a consequence of the fact that 
 $L^2 (Z, S_Z\otimes \mathcal{V}_{{\rm MF}})$ is isomorphic to the standard $C^*_r$-Hilbert module 
 $\mathcal{H}_{C^*_r \Gamma}$.
 
 From   \eqref{c-star-iso} we obtain
  \begin{equation}\label{c-star-k-iso}
  \begin{split}
K_*( C^* (V)^\Gamma)=K_*(C^* (G,S))\simeq K_* (\KK (\mathcal{E}))\simeq K_*(\KK (\mathcal{E}_{{\rm MF}}))\\\simeq K_*(\K \otimes 
 C^*_r \Gamma)\simeq K_* (C^*_r \Gamma)
 \end{split}
 \end{equation}
 where all the isomorphisms are canonical.
 
The K-theory of these $C^*$-algebras are the home of different equivalent definition of the index
class associated to the Dirac operator $D_V$. Let us recall these definitions
in  the even dimensional case:

\begin{itemize}
\item the coarse index class $\Ind^{{\rm coarse}} (D_V)\in K_0 (C^*
  (V)^\Gamma)$ of Definition \ref{def:coarse_ind};
\item the Connes-Skandalis index class $\Ind^{{\rm CS}} (D_V)\in K_0 (C^* (G,S_V))$ defined via the Connes-Skandalis
projector associated to a parametrix $Q$ for $D_V$. Thus, we first choose $Q$,
a $\Gamma$-compactly
supported pseudodifferential operator of order $(-1)$ so that
\begin{equation}\label{simbpar}
Q D^+_V =\Id-S_{+}\,,\quad\quad D^+_V Q=\Id - S_{-}
\end{equation}
with remainders $S_{-}$ and $S_{+}$ that are in 
$C^\infty_c (G, (s^*S_Z^\pm)^*\otimes r^* S_Z^\pm)$; then  we consider 
$\Ind^{{\rm CS}} (D_V):= [P_Q]-[e_1] \in K_0 (C^* (G,S_V))$, with 
\begin{equation}\label{e_Q}
P_Q:= \left(\begin{array}{cc} S_{+}^2 & S_{+}  (I+S_{+}) Q\\ S_{-} D^+_V &
I-S_{-}^2 \end{array} \right);\qquad e_1:=\left( \begin{array}{cc} 0 & 0 \\ 0&1
\end{array} \right).
\end{equation}
Here we have in fact defined the compactly supported index class $\Ind^{{\rm CS}}_c (D_V)$
$\in$ $K_0 (C^\infty_c (G, (s^*S_Z)^*\otimes r^* S_Z))$; the $C^*$-index class
is the image of this
class under the K-theory homomorphism induced by the inclusion
$C^\infty_c (G, (s^*S_Z)^*\otimes r^* S_Z)\hookrightarrow C^* (G,S_V)$. There are other 
equivalent descriptions of this  $C^*$-index class, such as the one defined by the
 Wassermann projector \cite[p.~356]{CM} or the graph
projector \cite[Section 8]{mn}; these are obtained from parametrices that are
of Sobolev order (-1) but are not
compactly supported.
\item the analogous  index class $\Ind^{{\rm CS}}_{{\rm MF}} (D_V)\in 
K_0 (\KK (\E_{{\rm MF}}))$, defined via a Mishchenko-Fomenko parametrix $\Q$ for $\D^+$,
with $\D^+$ equal to
the
Dirac operator on $Z$ twisted by $\mathcal{V}_{{\rm MF}}$;
\item the Mishchenko-Fomenko index class $\Ind_{{\rm MF}}( D_V):=[L_+] - [L_-]\in K_0 (C^*_r \Gamma)$,
defined via a Mishchenko-Fomenko decomposition induced by $\D^+$ on  $\E^1_{{\rm MF}}$ and 
$\E_{{\rm MF}}$;
\end{itemize}
 
 \begin{proposition}\label{prop:equality-index-closed}
 Up to canonical $K$-theory isomorphisms
 we have:
 \begin{equation}\label{equality-index}
 \Ind^{{\rm coarse}} (D_V)=\Ind^{{\rm CS}} (D_V)=\Ind^{{\rm CS}}_{{\rm MF}} (D_V)=
 \Ind_{{\rm MF}}( D_V)\,.
 \end{equation}
 \end{proposition}
 
  \begin{proof}
 The equality
 $\Ind^{{\rm CS}} (D_V)=\Ind^{{\rm CS}}_{{\rm MF}} (D_V)$ follows from 
 the second and third isomorphism in \eqref{c-star-iso}. A detailed proof appears in 
 \cite[Lemma 6.1]{CM}, where
  it is actually proved that $\Ind^{{\rm CS}}_{{\rm MF}} (D_V)$ is equal to the image of the 
  compactly supported  index class
  $\Ind^{{\rm CS}}_c (D_V)$ under the K-theory homomorphism induced by the inclusion
  $C^\infty_c (G, (s^*S_Z)^*\otimes r^* S_Z)\hookrightarrow C^* (G,S_V)$.\\
  For the other equalities we make a preliminary remark. It is clear that the
  Connes-Skandalis 
 index class of $D_V$ is equal to the Connes-Skandalis  index
 class of the 
 bounded transform of 
 $D_V$. This is true by general principles but can also be checked
 directly: consider 
 $A=  (1+D^2_V)^{-1/2} D_V$ and $B= (1+D_V^- D_V^+)^{1/2} Q$, which will be written
 shortly as $(1+D^2_V)^{1/2} Q$. Then $A^+ B= \Id-R_-$
 and $BA^+ = \Id-R_+$, with $R_-=S_-$ and $R_+= (1+D^2_V)^{1/2} S_+ (1+D^2_V)^{-1/2}$. 
 We can now write the Connes-Skandalis projector associated to $A^+$, $B$ and $R_\pm$
 and we call it $P_B$.
 This is homotopic  to the Connes-Skandalis projector \eqref{e_Q} (just consider 
 $(1+s D^2_V)^{\pm 1/2}$, with $s\in [0,1]$, throughout). Moreover, the K-theory class defined
 by $P_B$ is nothing but $\partial [A^+]$,
 the index class associated to $A^+$ via the short exact sequence
 $0\to \KK (\E)\to \BB (\E) \to \BB (\E)/\KK( \E)\to 0$. Now, the same remark applies
 to $\Ind^{{\rm CS}}_{{\rm MF}} (D_V)$ but for the isomorphic 
 short exact sequence $0\to \KK (\mathcal{E}_{{\rm MF}})\to \BB (\mathcal{E}_{{\rm MF}}) \to 
 \BB (\mathcal{E}_{{\rm MF}})/\KK( \mathcal{E}_{{\rm MF}})\to 0$.
We can now invoke the results in \cite[Section 17]{WO}, stating the equality of the latter index with the
Mishchenko-Fomenko index of $ (1+\D^2)^{-1/2} \D^+$. Using the Mishchenko-Fomenko
calculus the latter is easily seen to be the same as the Mishchenko-Fomenko
index of $\D^+$,  see for example \cite[Theorem 6.22]{Schick-ell2} for the
details. 

 Summarizing: we have also proved that 
$\Ind^{{\rm CS}}_{{\rm MF}} (D_V)=\Ind_{{\rm MF}} (D_V)$. It remains to show that  $\Ind^{{\rm coarse}} (D_V)=\Ind^{{\rm CS}} (D_V)$. To this end we start with the expression of the Connes-Skandalis index class
in terms of $A^+$, $B$, $R_\pm$: 
$$\Ind^{{\rm CS}} (D_V) = [P_B]-[e_1]\quad\text{with}\quad
 P_B=\left(\begin{array}{cc} R_{+}^2 & R_{+}  (I+R_{+}) B\\ R_{-} A^+ &
I-R_{-}^2 \end{array} \right).$$ Consider now the function $\chi (x):= x/\sqrt{1+x^2}$; this
is a chopping function (so, it can be used to define the coarse index)
 and $\chi (D_V)=A$. The inverse of $U^* A^+$ in 
$D^* (V)^\Gamma / C^* (V)^\Gamma$ can be represented by $B U\in D^*(V)^\Gamma$:
we observe that  
$$ (B U)(U^* A^+)= \Id - R_+\,; \quad (U^* A^+) (B U)= U^* (\Id - R_-) U\;.$$ We now write the 
expression of $\partial (U^* A^+)$ in terms of these choices and we get
$$ \Ind^{{\rm coarse}} (D_V)=[\Pi_B]- [e_1]\quad\text{with}\quad
 \Pi_B=\left(\begin{array}{cc} R_{+}^2 & R_{+}  (I+R_{+}) B U\\ U^* R_{-} A^+ &
U^* (I-R_{-}^2)U \end{array} \right).$$
Observing now that 
$$\Pi_B=  \left( \begin{array}{cc} 1 & 0 \\ 0&U^*
\end{array} \right) P_B \left( \begin{array}{cc} 1 & 0 \\ 0&U
\end{array} \right)$$
we conclude that $ \Ind^{{\rm coarse}} (D_V)= \Ind^{{\rm CS}} (D_V)$, as required.
  \end{proof}

\subsection{Index classes on manifolds with boundary}
We now pass to manifolds with boundary
and we adopt the notation explained in the geometric set-up
\ref{geometric-set-up}.
We remark that the complete Riemannian spin manifold $(W_\infty, g)$
 is in a natural way a $\Gamma$-covering of $Y_\infty$, the manifold with cylindrical 
end associated to $Y$. We consider the groupoid
$G_\infty:=W_\infty\times_\Gamma W_\infty$ 
and the associated  $C^*$-algebra $C^* (G_\infty)$, obtained 
by taking the closure of $C^\infty_c (G_\infty)$. Similarly, we can consider $C^* ( G_\infty,S)$
with $S$ the spinor bundle of $W_\infty$.
Then, as in \eqref{c-star-iso} 
\begin{equation}\label{equality-of-cstar}
\relC{W_\infty}{W}^\Gamma =C^* ( G_\infty,S)
\end{equation}

\begin{proposition}
  Assume that $W$ has dimension $n=2\ell$
  and that the boundary has positive scalar curvature.  Then, using a 
  $b$-parametrix, there is a well defined  
  $b$-index class associated to the spin Dirac operator on $W_\infty$:
 
  \begin{equation*}
    \Ind^b (D_{W_\infty})\in K_{0}(\relC{W_\infty}{W}^\Gamma ).
  \end{equation*}
    \end{proposition}
    
  \begin{proof}
  This class is, by definition, the 
  Connes-Skandalis projector associated to a suitable parametrix.
  We take the parametrix construction explained in \cite{mor-pia}
  (which is directly inspired by the $b$-parametrix construction of Melrose 
  \cite{Melrose}). The parametrix 
  in \cite{mor-pia} is explicitly proved to have remainders in $C^* ( G_\infty,S)$, thus, by
  \eqref{equality-of-cstar}, in $\relC{W_\infty}{W}^\Gamma$.
  \end{proof}

Together with this $b$-index class we consider other K-theory classes:

\begin{itemize}
\item the Mishchenko-Fomenko $b$-index class $\Ind^b_{{\rm MF}} (D_{W_\infty})\in K_0 (C^*_r \Gamma)$, obtained from a Mishchenko-Fomenko
decomposition theorem induced by a $b$-parametrix in the $b$-Mishchenko-Fomenko calculus on $Y_\infty$,
see \cite{LPMEMOIRS} and the Appendix of  \cite{LPBSMF};
\item the   index class $\Ind^{{\rm APS}}_{{\rm MF}} (D_W)\in K_0 (C^*_r \Gamma)$
defined by a Mishchenko-Fomenko boundary value problem on $Y$ \`a la Atiyah-Patodi-Singer, see \cite{WuI};
\item the   conic index class $\Ind^{{\rm conic}}_{{\rm MF}} (D_W)\in K_0 (C^*_r \Gamma)$
defined by a Mishchenko-Fomenko conic parametrix, see \cite{LPBSMF}.
\end{itemize}

\begin{remark}
Each one of these index classes has interesting features: $\Ind^b_{{\rm MF}} (D_{W_\infty})$
is the most suitable for proving higher index formulas (see the next subsection); the conic
index class displays the most interesting stability properties (see \cite{LLP}, where these two
properties
are used in order to define higher signatures on manifolds with boundary and proving,
under additional assumptions on $\Gamma$, their homotopy invariance). The APS index class,
on the other hand,
makes the study of gluing and cut-and-paste problems particularly easy.
\end{remark}


\begin{proposition}\label{prop:equality-aps}
Up to natural K-theory isomorphisms
  the following equalities hold:
  \begin{equation*}
    \Ind^{{\rm rel}}(D_{W_\infty})= \Ind^{b}(D_{W_\infty})  = 
    \Ind^b_{{\rm MF}} (D_{W_\infty})=\Ind^{{\rm APS}}_{{\rm MF}}
    (D_W)=\Ind^{{\rm conic}}_{{\rm MF}} (D_W) .
  \end{equation*}
\end{proposition}
\begin{proof}
The equalities $ \Ind^b_{{\rm MF}} (D_{W_\infty})=\Ind^{{\rm APS}}_{{\rm MF}} (D_W)=\Ind^{{\rm conic}}_{{\rm MF}} (D_W)$ are proved in \cite{LLP} (the case treated here, with full invertibility
of the boundary operator, is actually simpler than the one discussed in \cite{LLP}). The proof 
that $ \Ind^b (D_{W_\infty})= \Ind^b_{{\rm MF}} (D_{W_\infty})$ is as in the closed case (hence,
using the results in \cite{WO}). Thus we only
need to prove the first equality: $ \Ind^{{\rm rel}}(D_{W_\infty})= \Ind^{b}(D_{W_\infty})$.
However, this follows once again from the reasoning given in the proof of
Proposition \ref{prop:equality-index-closed}. Indeed, the relative-coarse
index class is defined in terms
of $U^*\chi (D_V)_+$, with $\chi$ a chopping function equal to $\pm 1$
on the spectrum of the boundary
operator. We have claimed that $\chi (D)_- U$ is an inverse of $U^*\chi (D_V)_+$ modulo 
$\relC{W_\infty}{W}^\Gamma$; now, in order to define $\pa [U^*\chi (D_V)_+]$ we can choose
an arbitrary inverse in the quotient. Consider $BU$, with $B=(1+D_V^2)^{1/2} Q$ 
and $Q$ equal to
a $b$-parametrix for $D_V^+$ as in \cite{mor-pia}. It is easy to see, from the expression of
$Q$,  that $BU$ is indeed an element in $D^* (W_\infty)^\Gamma$; moreover, since $S_\pm$,
the remainders given by $Q$, 
are residual terms in the $b$-calculus, it follows that $R_\pm$ are also residual, so that
$BU$ is indeed an inverse of $U^*\chi (D_V)_+$ mod $C^* ( G_\infty,S)$,
i.e. mod $\relC{W_\infty}{W}^\Gamma$. The proof now proceeds as in the closed case.
\end{proof}

\subsection{Delocalizing}
The goal of this subsection 
is to explain why we look at our index formula
\eqref{main-index-eq}
 \begin{equation*}
   \iota_* (\coarseind_\Gamma (D_W))  =\rho_\Gamma(g_{\pa W}) \quad \in\quad K_0 (D^*_\Gamma)
 \end{equation*}
as an equality between a \emph{delocalized} part of the index class  and the
rho-class
of the boundary operator. This is not needed in the sequel, but meant to put
our considerations in the context of previously established index theorems.

 To this end we recall the higher APS-index formula proved in \cite[Theorem
 14.1]{LPMEMOIRS}. This is a formula 
for the Karoubi Chern character  of the index class $\Ind^b (D_{W_\infty})$,
an element 
in the non-commutative de Rham homology $\overline{H}_* (\mathcal{B}^\infty)$ ($\mathcal{B}^\infty$ is a suitable
dense holomorphically closed subalgebra of $C^*_r \Gamma$, for example the Connes-Moscovici algebra).
The formula reads
\begin{equation}\label{higher-aps} 
\Ch (\Ind^b_{{\rm MF}} (D_{W_\infty}))=[ \int_{Y} \widehat{A} (Y,\nabla^Y)\wedge\omega -\frac{1}{2} 
\widetilde{\eta}(D_{\pa W})]\;\;\text{in}\;\; \overline{H}_* (\mathcal{B}^\infty)
 \end{equation}
 where $Y=W/\Gamma$, $\omega$ is a certain bi-form in 
 $\Omega^* (Y)\otimes \Omega_* (\CC \Gamma)$
 and where $$\widetilde{\eta}(D_{\pa W})\in \overline{\Omega}_* (\mathcal{B}^\infty):=
 \Omega_* (\mathcal{B}^\infty)/\overline{[ \Omega_* (\mathcal{B}^\infty),  \Omega_* (\mathcal{B}^\infty)]}
 $$ is 
 Lott's  higher eta invariant of the boundary operator \cite{LottII}, an invariant which is well defined for any 
 $L^2$-invertible Dirac operator $D_V$ on the total space of a boundaryless Galois 
 $\Gamma$-covering $V$  with base $Z$.
  Assume now, for simplicity, that  $\Gamma$ is {\it  virtually nilpotent.} Under this  additional assumption, the complex
$\overline{\Omega}_*(\mathcal{B}^\infty)$ splits as the direct sum of
sub-complexes labeled by the conjugacy classes of
$\Gamma$. Write $<\Gamma>$ for the set of conjugacy classes. Thus
$\widetilde{\eta}(D_V)$ splits as a direct sum
$$\widetilde{\eta}(D_V)=\bigoplus_{<x>\in <\Gamma>}
\widetilde{\eta}_{<x>} (D_V).$$
The higher $\rho$-invariant of Lott \cite{LottII}
associated to $D_V$ is, by definition,
$$\widetilde{\rho}(D_V):=\bigoplus_{<x>\not=
<e>}\widetilde{\eta}_{<x>} (D_V)$$

As pointed out in \cite[page 222]{LottII}, the higher $\rho$-invariant lies
in fact in $\overline{H}_*(\mathcal{B}^\infty)$. Notice that
$\overline{H}_*(\mathcal{B}^\infty)$ splits as the direct sum
$\overline{H}_*(\mathcal{B}^\infty)=\overline{H}_{<e>,*}(\mathcal{B}^\infty)\oplus 
\overline{H}_*^{{\rm deloc}}(\mathcal{B}^\infty)$ with the first group on the right hand side associated to the
subcomplex of $\overline{\Omega}_*(\mathcal{B}^\infty)$ labeled by the
trivial conjugacy class $<e>$ and the second group associated to the
subcomplexes labeled by the non-trivial conjugacy classes;
$\overline{H}_*^{{\rm deloc}}(\mathcal{B}^\infty)$ is thus the {\it delocalized} part of 
$\overline{H}_*(\mathcal{B}^\infty)$.
Then
$\widetilde{\rho}(D_V)$ lies in $\overline{H}_*^{{\rm
    deloc}}(\mathcal{B}^\infty)$.
 We let 
$\pi^{{\rm deloc}}$ be the  natural projection map. Then 
on the basis of the higher APS index formula \eqref{higher-aps}  one proves 
easily that
\begin{equation}\label{deloc-homology}
 \pi^{{\rm deloc}} (\Ch\circ \Ind^b_{MF} (D_{W_{\infty}}))= -\frac{1}{2} \widetilde{\rho} (D_{\pa W})\,.
 \end{equation}
Thus we see that Lott's higher rho invariant corresponds to the delocalized part of the Chern character
of the index class. We regard 
our equation $ \iota_* (\coarseind_\Gamma (D_W))  =\rho_\Gamma(g_{\pa W})$  as
a sharpening in K-theory of \eqref{deloc-homology}.

Finally, it is proved in \cite{LPPSC}
that the higher rho-invariant induces a well-defined group homomorphism
\begin{equation}\label{rho-hat}
\widehat{\rho}:\Pos^{\spin}_n (B\Gamma)\to \overline{H}_*^{{\rm deloc}}(\mathcal{B}^\infty)
\end{equation}
with 
$\widehat{\rho} [Z, u\colon Z\to B\Gamma, g_Z]= -\frac{1}{2} \widetilde{\rho} (D_{V})$, $V=u^* E\Gamma$,
and $*={\rm odd}$
if $n=2\ell$ and $*={\rm even}$ is $n=2\ell + 1$.
Together with \eqref{deloc-homology}
this shows that the the following diagram is commutative
\begin{equation}\label{diagramLP}
  \begin{CD}
   R^{\spin}_{n+1}(B\Gamma) @>>>
    \Pos^{\spin}_n (B\Gamma) \\ 
     @VV{\Ch\circ \Ind_\Gamma}V  @VV{\widehat{\rho}}V\\
      \overline{H}_* (\mathcal{B}^\infty) @>>> \overline{H}_*^{{\rm deloc}}(\mathcal{B}^\infty)\\
    \end{CD}
\end{equation}
In future work we plan to tackle the problem of defining a group
homomorphism
$\Ch_{{\rm deloc}}\colon  K_* (D^*_\Gamma)\to  \overline{H}_*^{{\rm deloc}}(\mathcal{B}^\infty)$ so that the following
double diagram is commutative and the composition of the two vertical arrows
on the right hand side is precisely $\widehat{\rho}$ in  \eqref{rho-hat}:
\begin{equation*}
  \begin{CD}
   R^{\spin}_{n+1}(B\Gamma) @>>>
    \Pos^{\spin}_n (B\Gamma) \\ 
     @VV{ \Ind_\Gamma}V  @VV{\rho}V\\
     K_{n+1} ( C^*_r\Gamma) @>>> K_{n+1}((D^*_\Gamma))  \\
     @VV{ \Ch}V  @VV{\Ch_{{\rm deloc}}}V\\
      \overline{H}_* (\mathcal{B}^\infty) @>>> \overline{H}_*^{{\rm deloc}}(\mathcal{B}^\infty)\\
    \end{CD}
\end{equation*}
In fact, we hope to map the whole Higson-Roe surgery sequence to a
sequence in non-commutative de Rham homology.

\section{$K$-theory homomorphisms}
\label{sec:K-theory-homs}

 \subsection{Geometrically induced homomorphisms}
Let $W$, $W_\infty$, $M=\pa W$ be as in the Geometric set-up
\ref{geometric-set-up}.

In the following, with slight abuse of notation we will denote all inclusions
$C^*(X)\to D^*(X)$ (and their equivariant and relative versions) by $\iota$
and the induced map in K-theory by $\iota_*$.

\begin{proposition}\label{def:iota}
We have commutative diagrams
\begin{equation}\label{iota}
  \begin{CD}
     K_*(C^*(\pa W)^\Gamma) @>{\iota_*}>> K_*(D^*(\pa W)^\Gamma)\\
  @VV{j_*}V @VV{j_*}V\\
   K_*(C^*(W)^\Gamma) @>{\iota_*}>> K_*(D^*(W)^\Gamma)\\
     @V{c}V{\iso}V @V{c}V{\iso}V\\
    K_* (\relC{W_\infty}{W}^\Gamma) @>{\iota_*}>> K_*
    (\relD{W_\infty}{W}^\Gamma)\,,\\
  \end{CD}
\end{equation}
\begin{equation}\label{eq:defjs}
  \begin{CD}
    K_*(D^*(\boundary W)^\Gamma) @>>> K_*(D^*(W)^\Gamma)\\
    @V{\iso}V{j_\boundary}V @V{\iso}V{c}V\\
    K_*(\relD{[0,\infty)\times\boundary W}{\boundary W}^\Gamma) @>{j_+}>>
    K_*(\relD{W_\infty}{W}^\Gamma)
  \end{CD}
\end{equation}
with $j_+$ and $j_{\pa}$ induced by the natural inclusions.
\end{proposition}
\begin{proof}
  To construct the maps $C^*(\pa W)^\Gamma\to C^*(W)^\Gamma$ and $D^*(\pa
  W)^\Gamma\to D^*(W)^\Gamma$ we can and will use the \emph{same} isometry
  covering the inclusion $j\colon \boundary W\to W$.
  Then all the maps are induced by inclusions of algebras and therefore the
  commutativity 
  follows from naturality of K-theory. The isomorphism claims in the statement
  have already been discussed. 
\end{proof}



\subsection{Kasparov Lemma}

The following Lemma, stated for the first time in \cite[Prop. 3.4]{kasparov-invariants}, see also
\cite[Lemma 7.2]{hig-ped-roe}, is a useful tool in proving  pseudolocality. 

\begin{lemma}[Kasparov Lemma]\label{lem:kasparov}
Let $H$ be an adequate $X$-module.
A bounded operator $A\colon H\to H$ is pseudo-local if and only if $\psi A \phi$ is compact 
whenever	$\psi$ and	 $\phi$ are bounded continuous functions on $X$ such that the supports 
${\rm supp}( \psi)$	and ${\rm supp}( \phi)$ are disjoint and at least one of
them is compact.

If $A$ is a norm limit of operators of bounded propagation, it is sufficient
to consider only functions of compact support.
\end{lemma}
\begin{proof}
A proof of the first statement for the case that $X$
is compact is given in \cite[5.4.7]{higson-roe-book}. It directly covers the
general case, as well.

The second statement about finite propagation operators is already remarked in
\cite[footnote 6]{hig-ped-roe}. To prove it, it suffices (by a limit
argument) to assume that $A$ has bounded propagation $R$. Then, given
$\phi,\psi$ with $\phi$ of 
compact support and such that $\phi\psi=0$, write $\psi=\psi_1+\psi_2$ such
that $\psi_1$ has compact support and $\psi_2$ has support of distance $R$
from $\phi$. Then by the bounded propagation property, $\phi A\psi_2=0$, and
by assumption $\phi A\psi_1$ is compact, so also $\phi A\psi$ is compact and
the assumptions of the usual form of the Kasparov Lemma are fulfilled.
\end{proof}

\subsection{The Mayer-Vietoris sequence}\label{subsect:mayer-vietoris}

Assume that $X=X_1\cup X_2$ is a Riemannian manifold (typically non-compact),
decomposed into two closed subsets $X_1,X_2$, with $X_0:=X_1\cap X_2$. (The
more general case of metric spaces is treated in exactly the same way.) 
We make  the following  \textbf{excision assumption}:
 $X_0:=X_1\cap X_2$ is big
enough in the following sense (of Higson-Roe-Yu \cite{hry-mv}):
for each $R>0$ there is $S>0$ such that $U_R(X_1)\cap U_R(X_2)\subset
U_S(X_0)$. 

\smallskip




Using along the way the relative Roe-algebras for $X_i\subset X$ and Lemma \ref{lemma:from-hry}, one finally
gets the expected commuting diagram of
6-terms exact Mayer-Vietoris sequences \cite[Section 3]{Siegel}
{\small\begin{equation}\label{mv}
  \begin{CD}
  \dots \to K_0(C^* (X_1))\oplus K_0(C^* (X_2)) @>>>
    K_0(C^*(X)) @>{\delta_{\rm{MV}}}>> K_1 (C^* (X_1 \cap X_2))   \to \\
   @VVV @VVV @VVV \\
   \dots \to K_0(D^* (X_1))\oplus K_0(D^* (X_2)) @>>>
    K_0(D^*(X)) @>{\delta_{\rm{MV}}}>> K_1 (D^* (X_1 \cap X_2))   \to\\
  \end{CD}
\end{equation}
}
Exactly the same works for the $\Gamma$-equivariant versions.

\subsection{Mayer-Vietoris for the cylinder}\label{subsect:mv-for-cylinder}

Let $M$ be a Galois $\Gamma$-cover of a compact manifold $Z$ and consider the cylinder $X:=\RR\times M$,
with $\Gamma$ acting in a trivial way on $\RR$. We set $X_1=(-\infty,0]\times M$ and 
$X_2= [0,\infty)\times M$ so that $X_0=\{0\}\times M=M$. 
This decomposition of the cylinder clearly satisfies the excision axiom;
thus we have the commuting diagram of long exact sequence for  $C^*$ and $D^*$ 
{\small 
\begin{equation}\label{mv-cyl-c}
    \begin{CD}
   \dots\to K_0(C^* ((-\infty,0]\times M)^\Gamma )\oplus
      K_0(C^* ([0,\infty)\times M)^\Gamma ) @>>> K_0(C^*(\RR\times
      M)^\Gamma) @>{\delta_{\rm{MV}}}>> K_1 (C^* (M)^\Gamma)\to \dots\\
  @VVV @VVV @VVV\\
   \dots\to K_0(D^* ((-\infty,0]\times M)^\Gamma )\oplus
      K_0(D^* ([0,\infty)\times M)^\Gamma ) @>>> K_0(D^*(\RR\times
      M)^\Gamma) @>{\delta_{\rm{MV}}}>> K_1 (D^* (M)^\Gamma)\to \dots
    \end{CD}
\end{equation}
}

\begin{lemma}
  For any metric space $M$ with isometric $\Gamma$-action
  \begin{equation*}
    K_*(C^*([0,\infty)\times M)^\Gamma) =0;\qquad K_*(D^*(
    [0,\infty)\times M)^\Gamma) =0.
  \end{equation*}
\end{lemma}
\begin{proof}
  This is a special instance of a general principle introduced by Roe: any
  space of the form $[0,\infty)\times M$ is \emph{flasque} in the sense of
  \cite[Definition 9.3]{roe-cbms}, and then by 
  \cite[Proposition 9.4]{roe-cbms} the K-theory of $C^*(
  [0,\infty)\times M)$ vanishes.

  The argument given there, based on an Eilenberg swindle, word by word
  applies also to the $\Gamma$-invariant subalgebras and to $D^*$.
\end{proof}

In particular, we conclude that the boundary maps in the Mayer-Vietoris
sequence induce compatible isomorphisms
\begin{equation}\label{iso-with-shift}
  \begin{CD}
     K_{*+1} (C^* (\RR\times
    M)^\Gamma) @>{\delta_{\rm{MV}}}>{\simeq}> K_{*}(C^*( M)^\Gamma)\\
   @VVV @VVV\\
     K_{*+1} (D^* (\RR\times
    M)^\Gamma) @>{\delta_{\rm{MV}}}>{\simeq}> K_{*}(D^*(M)^\Gamma) .
  \end{CD}
\end{equation}

\begin{remark}\label{rem:description_of_delta_MV}
 The maps  $\delta_{MV}$ of \eqref{iso-with-shift} are
  given as follows: take a representative $\alpha\in C^*(\reals\times
  M)^\Gamma$ of the K-theory class, i.e.~either a projector or an
  invertible (using that the algebra is stable). 
  One now has to trace the definition of the splicing argument
  which gives rise to \eqref{mv}: map $\alpha$ to
  $C^*(\reals\times M)^\Gamma/\relC{\reals\times M}{
    (-\infty,0]\times M}^\Gamma$. Then we lift it (through the inverse of the
 natural  isomorphism) to $C^*( [0,\infty)\times M)^\Gamma/\relC{
   [0,\infty)\times M}{ \{0\}\times M}^\Gamma$. For this, let $\psi_+$ be the
  characteristic function of $
  [0,\infty)\times M$, then the compression $\psi_+\alpha\psi_+$ (where $\psi_+$ acts
  as multiplication operator) is such a lift. That it has the relevant
  properties follows as in Lemmas \ref{lem:involutions_in_quotient} and
  \ref{lem:compact_supp} proved below.
Finally, $\delta_{MV}(\alpha)=c^{-1}\delta([\psi_+\alpha\psi_+])$ where
$\delta$ is 
  the boundary homomorphism of the K-theory long exact sequence for
  $\relC{[0,\infty)\times M}{ \{0\}\times M}^\Gamma\into C^*(
  [0,\infty)\times M)^\Gamma$ and $c$ is the isomorphism of Lemma
  \ref{lemma:from-hry}. 

  Exactly the same construction works for $D^*$.
\end{remark}



\section{Proofs of the main theorems} 
\label{sec:proofs}

The goal of this Section is to provide a proof of our two main theorems.
We shall begin by stating a key result, the ``cylinder delocalized index
theorem''. This result
is the cornerstone for the proof of both  theorems. We state the
cylinder delocalized index theorem in Subsection
\ref{subsect:key-cylinder}, but we defer the (quite technical) proof to
Subsection \ref{sec:proof_cylindercase}.
Next we explain how the cylinder delocalized index theorem can be employed in
order to prove both theorems. We do this in Subsection \ref{reduction-aps}
and Subsection \ref{reduction-partition}. Finally, as anticipated, we give  a 
detailed proof of the cylinder delocalized index theorem in Subsection \ref{sec:proof_cylindercase}.

\subsection{The cylinder delocalized index
  theorem}\label{subsect:key-cylinder}

\begin{notation}
We let $M$ be a boundaryless manifold with 
a free, isometric and cocompact action of $\Gamma$. 
We assume that $M$ is endowed with a $\Gamma$-invariant metric
of positive scalar curvature. We assume $M$ to be of dimension  $n$,
with $n$ odd.
We shall consider
$
\RR\times M$, 
$\RR_{\geq 0}\times M$,
$\RR_{\leq 0}\times M$.
We consider the Dirac operators
$$D_M \;\;\text{on} \;\;M \;\;\;\;\text{and}\;\;\;\; D_{\cyl} \;\;\text{on} \;\; \RR\times M\,.$$
We shall also employ the notation $D_{\RR\times M}$ for  $D_{\cyl}$.
\end{notation}

The positive scalar curvature assumption on $M$ implies that $D_M$
is $L^2$-invertible; thus there is a well defined $\rho$-class
$\rho (D_M)\in K_0 (D^* (M)^\Gamma)$. Also $D_{\cyl}$ is
$L^2$-invertible; 
hence $\chi (D_{\cyl})$, with $\chi$ a suitable chopping function, is an involution. 
This means that there is a well defined $\rho$-class on the cylinder:
$\rho(D_{\RR\times M})\in 
K_{1} (D^* (\RR\times M)^\Gamma)$. We know by \eqref{iso-with-shift} that for
the cylinder
$\RR\times M= (\RR_{\leq}\times M)\cup_M (\RR_{\geq}\times M)$
there is a well-defined  Mayer-Vietoris isomorphism 
$$\delta_{{\rm MV}} : K_{1} (D^* (\RR\times M)^\Gamma)\to K_0 (D^* (M)^\Gamma).$$
 Thus it makes sense to consider $\delta_{{\rm MV}}(\rho(D_{\RR\times M}))
\in K_0 (D^* (M)^\Gamma)$. The following result will be crucial:

\begin{theorem}[Cylinder delocalized index
  theorem]\label{theo:fundamental-cylinder}
\begin{equation}\label{fundamental-cylinder}
\delta_{{\rm MV}} (\rho(D_{\RR\times M}))= \rho (D_M) \quad\text{in}\quad K_0 (D^* (M)^\Gamma).
\end{equation}
\end{theorem}

\subsection{Proof of the delocalized APS index theorem assuming Theorem \ref{theo:fundamental-cylinder}}
\label{reduction-aps}

In this subsection we make use of the fundamental identity on the cylinder,
 \eqref{fundamental-cylinder},
in order to give a proof of Theorem \ref{theo:k-theory-deloc},  the delocalized APS index theorem.

\begin{notation}
We consider $W$, $\pa W$ and $W_\infty$ as in the geometric set-up 
\ref{geometric-set-up}. 
We also consider 
$$
\RR\times \partial W\,,\quad 
\RR_{\geq 0}\times \partial{W}\,,
\quad 
\RR_{\leq 0}\times \partial{W}\,.
$$
We consider the Dirac operators
$$D \;\;\text{on} \;\; W_\infty\,,\;\;D_\partial \;\;\text{on} \;\;\partial{W}\;\;\text{and}\;\;D_{\cyl} \;\;\text{on} \;\; \RR\times \partial{W}\,.$$
Recall that the boundary $\partial{W}$ is endowed with a metric of positive
scalar curvature, so that
the coarse index class of $D$ is well defined as an element  $\Ind^{{\rm rel}}
(D)\in 
K_0 ( \relC{W_\infty}{W}^\Gamma)$. 
The positive scalar curvature assumption implies that also $D_{\cyl}$ is
$L^2$-invertible; 
hence $\chi (D_{\cyl})$ is an involution, too.  

 We denote by $\psi$ the characteristic function
of $ [0,\infty)\times \boundary W$ on $W_\infty$ and by $\psi_+$ the
corresponding characteristic function on $\RR\times \partial{W}$.
\end{notation}

Consider the operator $\psi_{+} \,\chi (D_{\cyl})\,\psi_{+}$
on $\RR_{\geq}\times \partial{W}$; obviously, this is not an involution any more. 
Similarly,  consider the operator $\psi \,\chi (D_{\cyl})\,\psi$ on
$W_\infty$, which also fails to be an involution.

We start with a basic lemma about commutators with $\psi_+$.
\begin{lemma}\label{lem:compact_supp}
  If $T\in D^*(\RR\times \partial W)^\Gamma$ then
  $[T,\psi_+]$ is in 
  $\relD{\RR\times \partial W}{\partial W}^\Gamma$, and correspondingly if
  $T\in D^*(W_\infty)^\Gamma$ then $[T,\psi]\in \relD{W_\infty}{W}^\Gamma$. 
\end{lemma}
\begin{proof}
  We follow the proof of \cite[Lemma 4.3]{roe-cbms}. We already know that 
 $[T,\psi_+]$ belongs to  $D^*(\RR \times \partial W)^\Gamma$
 and we only need to show that it lies actually in the 
 ideal  $\relD{\RR\times \partial W}{\partial W}^\Gamma$.
  We can assume that $T$ has finite propagation $R$. Then, outside a
  sufficiently large neighborhood of the support of $\psi_+$, $[T,\psi_+]$ is
  zero, because there $\psi_+$ acts as the identity. 
  It follows that it is compactly supported in the $\RR$ direction, as desired.

  Finally, given $\phi\in C_c((0,\infty)\times\partial W)$, we have to show
  that $[T,\psi_+]\phi$ is compact. But
  \begin{equation*}
    [T,\psi_+]\phi = T\phi-\psi_+T\phi= (1-\psi_+)T\phi.
  \end{equation*}
  Because of finite propagation of $T$ we can replace $(1-\psi_+)$ by
  $(1-\psi_+)\alpha$ where $\alpha$ has compact support. Then, as
  $(1-\psi_+)\alpha\phi=0$, by the pseudolocality of $T$ this operator
  $(1-\psi_+)T\phi=(1-\psi_+)\alpha T\phi$ indeed is compact.

 \end{proof}

\begin{remark}\label{remark-general1}
The first part of Lemma \ref{lem:involutions_in_quotient} holds unchanged if we
consider, more generally, a boundaryless manifold $M$ with 
a free, isometric and cocompact action of $\Gamma$ and endowed with a $\Gamma$-invariant
metric of positive scalar curvature. In this case
  $\psi_{+} \,\chi (D_{\cyl})\,\psi_{+}$ is an involution in 
  $D^*(\RR_{\geq}\times M)^\Gamma/ \relD{\RR_{\geq}\times 
  M}{M}^\Gamma$.
  \end{remark}

\begin{lemma}\label{lem:involutions_in_quotient}
  $\psi_{+} \,\chi (D_{\cyl})\,\psi_{+}$ is an involution in 
  $D^*(\RR_{\geq}\times \partial{W})^\Gamma/ \relD{\RR_{\geq}\times 
  \partial W}{\partial W}^\Gamma$, where we write briefly $\relD{\RR_{\geq}\times 
  \partial W}{\partial W}^\Gamma$ instead of $\relD{\RR_{\geq}\times 
  \partial W}{\{0\}\times \partial W}^\Gamma$.
Similarly, $\psi \,\chi (D_{\cyl})\,\psi$ is an involution in 
$D^*(W_\infty)^\Gamma/\relD{W_\infty}{W}^\Gamma$. 
\end{lemma}

\begin{proof}
  We choose $\chi$ such that $\chi(D_{\cyl})^2=1$, this is possible because
  $0$ is not in the spectrum of $D_{\cyl}$ by the positive scalar curvature
  assumption. Then, using that $\psi_+^2=\psi_+$
  \begin{equation*}
    \begin{split}
      (\psi_+\chi(D_{\cyl})\psi_+)^2 &= \psi_+\chi(D_{\cyl})^2\psi_+ +
      \psi_+\chi(D_{\cyl})[\psi_+,\chi(D_{\cyl})]\psi_+\\
      &= 1 + (\psi_+-1) + \psi_+\chi(D_{\cyl})[\psi_+,\chi(D_{\cyl})]\psi_+.
    \end{split}
  \end{equation*}
  Observe that the second and the third operator are both in
  $D^*(W_\infty)^\Gamma$. Note that, on $[0,\infty)\times\boundary W$,
  $\psi_+-1=0$. On $W_\infty$, the corresponding $\psi-1$ is the negative
  of the characteristic 
  function of $W$, so has propagation $0$ and vanishes identically on
  $[0,\infty)\times \boundary W$, therefore $(\psi-1)\in
  \relD{W_\infty}{W}^\Gamma$. 
Using  Lemma \ref{lem:compact_supp}, we see that also the third
summand  belongs to $
  \relD{W_\infty}{ W}^\Gamma$ or $\relD{[0,\infty)\times\boundary W}{\boundary
    W}^\Gamma$, so the statement
  follows.
\end{proof}



\begin{definition}
Let 
$n+1$, the dimension of $W$, be even.\\ 
Consider  the half cylinder $\RR_{\geq} \times \boundary W$;
observe that  the spinor bundle is in this case the pull-back of the direct sum of two copies of the spinor bundle
  on $\boundary W$.  Thus, in this case, we could choose $U$ to
  be the identity. Using Lemma \ref{lem:involutions_in_quotient} we can define 
  the class 
 \begin{equation}\label{cut-index-cyl}
 [U^* (\psi_+ \chi(D_{\cyl})_+ \psi_+)]  \in
  K_1(D^* (\RR_{\geq}\times \boundary W)^\Gamma/ 
  \relD{\RR_{\geq}\times \boundary W}{\boundary W}^\Gamma)
  \end{equation}
  and thus, applying the boundary map
  for the obvious 6-terms long exact sequence
 \begin{equation}\label{bdry_map_for-cyl}
 \partial\colon  K_{1} (D^* ( \RR_{\geq}\times \partial{W})^\Gamma / \relD{\RR_{\geq}\times 
   \partial{W}}{\partial{W}}^\Gamma) 
\rightarrow K_{0} ( \relD{\RR_{\geq}\times 
   \partial{W}}{\partial{W}}^\Gamma),
\end{equation}
we obtain a class 
 \begin{equation}\label{cut-index-cyl-bis}
\pa [U^* (\psi_+ \chi(D_{\cyl})_+ \psi_+ )] \in K_{0} ( \relD{\RR_{\geq}\times 
   \partial{W}}{\partial{W}}^\Gamma).
\end{equation}
 Similarly, if $n+1$ is odd then 
  we  have a well defined class
  \begin{equation}\label{cut-index-cyl-bis-bis}
  \pa (\ha [1+\psi_+ \chi(D_{\cyl})\psi_+])\;\in \; K_1 (\relD{\RR_{\geq}\times 
   \partial{W}}{\partial{W}}^\Gamma)\end{equation} 
 \end{definition}
 
 \begin{remark}\label{remark-general2}
 There is a corresponding statement, obtained by replacing $\partial{W}$
 by a general $M$ with positive scalar curvature. In particular if $M$ is odd dimensional, 
 choosing  $U$ to be the identity, this gives classes
 \begin{equation}\label{cut-index-cyl-general} 
 [ \psi_+ \chi(D_{\cyl})_+ \psi_+ ]  \in
  K_1(D^* (\RR_{\geq}\times M)^\Gamma/ 
  \relD{\RR_{\geq}\times M}{M}^\Gamma)  \end{equation}
and 
 \begin{equation}\label{cut-index-cyl-bis-general}
\pa [ \psi_+ \chi(D_{\cyl})_+ \psi_+ ]  \in
  K_0 (  \relD{\RR_{\geq}\times M}{M}^\Gamma).
  \end{equation}
   \end{remark}
 
 \begin{remark}\label{remark:explicit_mv}
 We can restate Lemma \ref{lem:involutions_in_quotient}, and its obvious
 extensions in Remarks  \ref{remark-general2} and \ref{remark-general1} in a
 more conceptual way. Indeed,
 exactly the same proof establishes the following statements: \\
If $n+1$ is even then  {\it compression by $\chi_+$ gives a well
   defined homomorphism} 
 $$K_1 (D^* (\RR\times M)^\Gamma) \to 
 K_1 (D^* (\RR_{\geq}\times M)^\Gamma / 
 \relD{\RR_{\geq}\times M}{M})^\Gamma) $$
  {\it which sends the $\rho$-class defined by $D_{\RR\times M}$ to the class \eqref{cut-index-cyl-general}}.
 
  \smallskip
  \noindent
 Therefore, composition of this homomorphism with the boundary map 
  \begin{equation*}
 \partial\colon  K_{1} (D^* ( \RR_{\geq}\times M)^\Gamma / \relD{\RR_{\geq}\times 
 M}{M}^\Gamma) 
\rightarrow K_{0} ( \relD{\RR_{\geq}\times 
   M}{M}^\Gamma)
\end{equation*}
  gives a homomorphism 
 $$H\colon K_1 (D^* (\RR\times  M)^\Gamma)
 \to K_{0} ( \relD{\RR_{\geq}\times M}{M}^\Gamma)\,.$$
 Further composing with the inverse of the isomorphism 
 $$ j_{M}:  K_{0} ( D^*(M)^\Gamma)\to K_{0} ( \relD{\RR_{\geq}\times M}{M}^\Gamma)$$
 induced by the inclusion (it is the the analogue of \eqref{eq:defjs})
 gives finally  a well defined homomorphism:
 \begin{equation}\label{altervative_mv}
 K_1 (D^* (\RR\times M)^\Gamma)\to K_{0} ( D^*(M)^\Gamma).
 \end{equation}
 and this homomorphism  sends  $\rho (D_{\RR\times M})$
 into $j_{M}^{-1} \pa [ \psi_+ \chi(D_{\cyl})_+ \psi_+ ] $.
 It is not difficult to show, proceeding exactly as in Remark \ref{rem:description_of_delta_MV}, that the
 homomorphism \eqref{altervative_mv} is precisely the Mayer-Vietoris homomorphism $\delta_{{\rm MV}}$
 we have described in Subsection \ref{subsect:mv-for-cylinder}; in particular, by the argument given
 in Subsection \ref{subsect:mv-for-cylinder},  the homomorphism 
 \eqref{altervative_mv} is an isomorphism. Moreover, by the above remarks, the following identity
 holds in $K_{0} ( D^*(M)^\Gamma)$:
 \begin{equation}\label{fundamental-partial}
 \delta_{{\rm MV}} (\rho (D_{\RR\times M}))= j_{M}^{-1} \pa [ \psi_+ \chi(D_{\cyl})_+ \psi_+ ] 
 \end{equation}
The corresponding statement holds in the odd dimensional case.
 \end{remark}

We now go back to the manifold with cylindrical ends $W_\infty$.
Then, by the second part of Lemma \ref{lem:involutions_in_quotient} we have, in the even dimensional
case,
\begin{equation*}
[U^* (\psi \,\chi (D_{\cyl})_+\,\psi)]\in K_{1} (  D^* (W_\infty)^\Gamma / \relD{W_\infty}{W}^\Gamma) 
\end{equation*}
and thus applying the boundary map
 $$\partial\colon K_{1} (  
D^* (W_\infty)^\Gamma / \relD{W_\infty}{W}^\Gamma)\to
 K_{0} ( \relD{W_\infty}{W}^\Gamma )$$ 
 we  obtain  an element
\begin{equation}\label{cut-index-cyl-end}
\partial [U^* (\psi \,\chi (D_{\cyl})_+\,\psi)]\in K_{0} ( \relD{W_\infty}{W}^\Gamma ).
\end{equation}
In the odd dimensional case we have a corresponding class
\begin{equation}\label{cut-index-cyl-end-bis}   
\pa (\ha [1+\psi \chi(D_{\cyl})\psi])\;\in\; K_{1} ( \relD{W_\infty}{W}^\Gamma ).
\end{equation}



   
\begin{lemma}\label{lem:subtract}
  $\chi(D)-\psi\chi(D_{\cyl})\psi \in \relD{W_\infty}{W}^\Gamma$.
\end{lemma}
\begin{proof}
  This is a consequence of the proof of \cite[Proposition
  1.5]{roe-relative}. We have to show two things:
  \begin{enumerate}
  \item\label{item:compact} if $\phi\in C_0(W_\infty)$ with $d(W,\supp(\phi))>\epsilon>0$ then
    $(\chi(D)-\psi \chi(D_{\cyl}\psi))\phi$ is compact
  \item\label{item:propa} For each $\epsilon>0$ there is $R>0$ such that
    $\norm{(\chi(D)-\psi 
      \chi(D_{\cyl}\psi))\phi} <\epsilon$ whenever $\phi\in C_0(W_\infty)$ with $d(W,\supp(\phi))>R$.
  \end{enumerate}

Write the distribution $\hat\chi(t)=\alpha(t)+\beta(t)$ such that $\alpha$ is
smooth and rapidly decreasing and $\beta$ is a distribution supported in
$(-\epsilon,\epsilon)$. Then $\chi(D) = \int_{\reals} \alpha(t) e^{itD} +
\int_{-\epsilon}^\epsilon \beta(t) e^{itD}$. By unit propagation and isometry
invariance of the wave operator, $e^{itD}\phi = e^{itD_{\cyl}}\phi$ for $t\le
R$ if $d(W,\supp(\phi))>R$.
Therefore, in this situation
\begin{equation*}
(  \chi(D)-\psi\chi(D_{\cyl}))\phi = \int_{\abs{t}\ge R} \alpha(t)
\left(e^{itD}-\psi e^{it D_{\cyl}}\right)\phi \qquad\implies\quad \norm{(
    \chi(D)-\psi\chi(D_{\cyl}))\phi} \le 2\sup_{\abs{t}\ge R} \abs{\alpha(t)}.
\end{equation*}

As $\alpha$ is smooth and rapidly decreasing, the operator belongs to
$C^*(W_\infty)^\Gamma$ and therefore is compact for each $R>0$, and the norm
converges to $0$ as $R\to \infty$. This establishes the two properties.
\end{proof}


\begin{lemma}
Let $\iota_* \colon K_*  (\relC{W_\infty}{W}^\Gamma )
\longrightarrow K_*  (\relD{W_\infty}{W}^\Gamma )$ be the homomorphism induced
by the inclusion (we considered $\iota_*$ in Lemma
\ref{def:iota}). If $n+1$ is even then
\begin{equation}\label{iota-pa-1}
\iota_* ( \Ind^{\rm rel} (D))= \partial [U^* (\psi \,\chi (D_{\cyl})_+\,\psi)] \text{ in }
K_{0}  (\relD{W_\infty}{W}^\Gamma)) \;.
\end{equation}
If $n+1$ is odd, then 
\begin{equation}\label{iota-pa-1-bis}
\iota_* ( \Ind^{\rm rel} (D))= \partial [\ha (1+ \psi \,\chi (D_{\cyl})\,\psi)] \text{ in }
K_{1}  (\relD{W_\infty}{W}^\Gamma ) \;.
\end{equation}
\end{lemma}

\begin{proof}
 Let us prove the case in which $n+1$ is even.
 Recall the inclusion of the ideal $\relC{W_\infty}{W}^\Gamma\subset \relD{W_\infty}{W}^\Gamma $.
 Using an obvious commutative diagram we see that the left hand side
 of \eqref{iota-pa-1} is nothing but the boundary map
 applied to the involution
 \begin{equation*}
[U^*\chi(D)_+)] \in K_1(D^*(W_\infty)^\Gamma
 /\relD{W_\infty}{W}^\Gamma)).
\end{equation*}
The lemma 
 follows immediately from Lemma \ref{lem:subtract}. The odd  case is similar.
\end{proof}


For the next lemma and the following proposition recall the homomorphisms
$j_+$ and $j_{\pa}$ appearing in Proposition \ref{def:iota}. 
\begin{lemma}
 If $n+1$ is even then 
\begin{equation}\label{iota-pa-3}
\partial [U^* (\psi \,\chi (D_{\cyl})_+\,\psi)] = j_+ (\pa [U^* (\psi_+ \chi(D_{\cyl})_+ \psi_+)]) \text{ in } K_{n+1}  
(  \relD{W_\infty}{W}^\Gamma  ) .
  \end{equation}
  If $n+1$ is odd
  \begin{equation}\label{iota-pa-3-bis}
\partial [\ha (1+ \psi \,\chi (D_{\cyl})\,\psi)] = j_+ (\pa [\ha (1 + \psi_+
\chi(D_{\cyl}) \psi_+)]) \text{ in } K_{n+1}  (  \relD{W_\infty}{W}^\Gamma  ) .
  \end{equation}
\end{lemma}

\begin{proof}
We only prove \eqref{iota-pa-3}, the other statements can be derived
similarly. We choose as adequate modules on $\RR_{\geq} \times \pa W$ and on
$W_\infty$ the 
$L^2$-sections of the corresponding spinor bundles. Observe now that the
adequate module  
for $\RR_{\geq} \times \pa W$ is a direct summand of the one for $W_\infty$.
Thus, we can choose the isometry $V$ covering the inclusion $\RR_{\geq} \times \pa W
\hookrightarrow W_\infty$ to be simply given by the inclusion of the first module
as a direct summand into the second.
Then, by definition, $ j_+ ([U^* (\psi_+ \chi(D_{\cyl})_+ \psi_+)])= [U^* (\psi \,\chi (D_{\cyl})_+\,\psi)]$
and since $j_+$ commutes with the boundary map, we are done.
\end{proof}
Notice that we can choose $U$ to induce the identity on the cylindrical end,
where the positive and the negative spinor bundles are both the pullback of
the spinor bundle on $\boundary W$. Then the above
identity reads 
\begin{equation}\label{iota-pa-3-noU}
\partial [U^* (\psi \,\chi (D_{\cyl})_+\,\psi)] = j_+ (\pa [ \psi_+ \chi(D_{\cyl})_+ \psi_+ ])
\end{equation}

Finally, we have the crucial 

\begin{proposition}
  \label{proposition:crucial}
If $n+1$ is even then
\begin{equation}\label{iota-pa-2}
\pa [ \psi_+ \chi(D_{\cyl})_+ \psi_+]=j_\partial\, \rho(D_\partial)  \text{ in } 
K_{n+1} ( \relD{\RR_{\geq}\times
\partial{W}}{\partial{W}}^\Gamma ).
\end{equation}
\end{proposition}

\begin{proof}
  We apply $j_\partial^{-1}$ to both sides.
Using \eqref{fundamental-partial} for  $M=\pa W$ we see that
\eqref{iota-pa-2} is equivalent to
\begin{equation}\label{partioned-rho-cylinder-0}
\delta_{{\rm MV}} (\rho (D_{\RR\times \pa W}))=  \rho(D_{\pa W}) \text{ in } K_{n+1} ( D^* (\pa W)^\Gamma)
\end{equation}
which is precisely the content of Theorem \ref{theo:fundamental-cylinder}
(this is the cylinder delocalized index theorem).
\end{proof}

\begin{remark}
  Of course, we expect that the corresponding formula to
  Equation \eqref{iota-pa-2} holds if $n+1$ is odd, namely
\begin{equation}\label{iota-pa-2-bis}
\pa [\ha (1+ \psi_+ \chi(D_{\cyl}) \psi_+)]=j_\partial\, \rho(D_\partial)  \text{ in } K_{n+1} ( \relD{\RR_{\geq}\times
\partial{W}}{\partial{W}}^\Gamma ).
\end{equation}
\end{remark}

\smallskip
\noindent
{\bf Proof of Theorem \ref{theo:k-theory-deloc}}\\
We can finally give the proof of Theorem \ref{theo:k-theory-deloc}.
 Indeed, if $n+1$ is even then
from \eqref{iota-pa-1}, \eqref{iota-pa-3-noU}, \eqref{iota-pa-2} we obtain at once
 \begin{equation*}
\iota_* ( \Ind^{\rm rel} (D))= \partial [U^*(\psi \chi (D_{\cyl})_+ \psi)= j_+ (\partial [ \psi_+ \chi (D_{\cyl})_+ \psi_+ ])
=j_+ (j_\partial\, \rho(D_\partial) ).
\end{equation*}
Applying $c^{-1}$ and using the commutativity of \eqref{eq:defjs} we get
precisely what we have to show.



\subsection{Proof of the partitioned manifold theorem for $\rho$-classes
assuming Theorem \ref{theo:fundamental-cylinder}
}\label{reduction-partition}

In this subsection we show how to prove Theorem \ref{theo:part_mf_abst}
assuming the cylinder delocalized index theorem
\ref{theo:fundamental-cylinder}, namely that
\begin{equation*}\delta_{{\rm MV}} (\rho(D_{\RR\times M}))= \rho (D_M) \quad\text{in}\quad K_0 (D^* (M)^\Gamma).
\end{equation*}

Consider $(W,g)$, an $(n+1)$-dimensional Riemannian manifold 
with uniformly positive scalar curvature metric $g$,
partitioned by a two-sided hypersurface $M$, $W=W_- \cup_M W_+$, with product
structure near $M$ and with signed distance function $f\colon W\to \reals$.
We also
assume an isometric action of $\Gamma$, preserving $M$
and with the property that $M/\Gamma$ is compact. There is then a  resulting $\Gamma$-map
$u\colon M\to E\Gamma$.
We have defined in Subsection \ref{subsect:intro-partitioned} the partitioned manifold 
$\rho$-class $\rho^{{\rm pm}}(g)\in K_n (D^* (M)^\Gamma)$ and the partitioned manifold 
$\rho_\Gamma$-class $\rho_\Gamma^{{\rm pm}}(g)\in
K_n (D^*_\Gamma)$. 
We shall also employ the notation $\rho_\Gamma^{{\rm pm}}(D_W)$ for this class.

Recall that our goal is to show that 
\begin{equation}\label{restate-pm}
\rho_\Gamma^{{\rm pm}}(g)= \rho_\Gamma(g_M)\quad\text{in}\quad 
K_n (D^*_\Gamma)
\end{equation}
We first show that the left hand side is unchanged if we replace $W$
by $ \reals\times M$.
Notice that our proof only applies to $\rho_\Gamma^{{\rm pm}}(g)\in
K_n (D^*_\Gamma)$; it does not apply to $\rho^{{\rm pm}}(g)\in K_n (D^* (M)^\Gamma)$.

First of all, we extend our discussion in Remark \ref{rem:description_of_delta_MV}
and give a more detailed description of $\delta_{{\rm MV}} [D_W]$, 
with $\delta_{{\rm MV}}$ the Mayer-Vietoris boundary homomorphism  associated
to the partition $W=W_- \cup_M W_+$. To this end 
we recall that $\delta_{{\rm MV}}\colon K_n (D^* (W)^\Gamma)\to K_{n+1}
(D^*(M)^\Gamma)$  
is obtained by composing
\begin{equation*}
\begin{split}K_n (D^* (W)^\Gamma)\to K_n (D^* (W)^\Gamma/ \relD{W}{W_-}^\Gamma)\simeq\\
 K_n (\relD{W}{W_+}^\Gamma)/\relD{W}{M}^\Gamma) \xrightarrow{\pa} K_{n+1}
 (\relD{W}{M}^\Gamma)\xrightarrow[\iso]{c^{-1}} K_{n+1}(D^*(M)^\Gamma)
 \end{split}\end{equation*}
Let $\chi_{\pm}$ be the characteristic
functions of $W_{\pm}$. By writing an element $x$ in $K_n (D^* (W)^\Gamma)$ as $\chi_+ x \chi_+ + \chi_- x \chi_- + R$,
with $R\in \relD{W}{M}^\Gamma$, we see first of all that $\delta_{{\rm MV}} (\rho (D_W))$,
with $\rho (D_W)\in K_n (D^* (W)^\Gamma)$,
is equal to $c^{-1}\pa  [\chi_+ \rho (D_W)  \chi_+]$ where $\pa $ is equal to the
connecting homomorphism
for the ideal $\relD{W}{M}^\Gamma $ in $\relD{W}{W_+}^\Gamma$.
 Observe now that there is an isomorphism
of algebras 
\begin{equation}\label{iso-part}
 \relD{W}{W_+}^\Gamma/\relD{W}{M}^\Gamma \xleftarrow{\alpha}  D^* (W_+)^\Gamma/\relD{W_+}{M}^\Gamma
 \end{equation}
and that the lift of 
  $[\chi_+ \rho  (D_W)  \chi_+] \in  \relD{W}{W_+}^\Gamma/\relD{W}{M}^\Gamma$
   through  $\alpha$ is the class $[\chi_+ \rho  (D_W)  \chi_+]\in D^* (W_+)^\Gamma/\relD{W_+}{M}^\Gamma$,
   i.e.~the same element seen in a different algebra. Of course this correspondence will hold also
   for the associated K-theory elements.
Consider now the manifold $W_{\cyl,+}:=(\RR_{\leq}\times M)\cup_M  W_+ $; thus we cut out $W_-$ and
we glue at its place $\RR_{\leq}\times M$.  We can also consider the class $[\chi_+ (D_{\cyl,+}) \chi_+]  \in 
K_n (\relD{W_{\cyl,+}}{W_+}^\Gamma/\relD{W_{\cyl,+}}{M}^\Gamma)$
and its lift to $K_n (D^* (W_+)^\Gamma/\relD{W_+}{M}^\Gamma)$ under the K-theory isomorphism
induced by the analogue  to \eqref{iso-part} but for $W_{\cyl,+}$. The two lifts can now be compared,
as they live in the K-theory of the same algebra, which is $D^*
(W_+)^\Gamma/\relD{W_+}{M}^\Gamma$.

\begin{lemma}\label{lem:reduction-partitioned}
In $K_n (D^* (W_+)^\Gamma/\relD{W_+}{M}^\Gamma)$ the following equality holds:
\begin{equation}\label{reduction-partitioned}
[\chi_+ D_W \chi_+]=
[\chi_+ (D_{\cyl,+}) \chi_+]
\end{equation}
\end{lemma}

\noindent
Assuming the lemma we now conclude the proof of
 the partitioned  manifold theorem for $\rho$-classes.\\
Theorem \ref{theo:fundamental-cylinder}, the cylinder delocalized index theorem,
 states that for any  $n$-dimensional 
$\Gamma$-manifold without boundary, $n$ odd,  with isometric, free, cocompact action
and positive scalar curvature one has 
$$\delta_{{\rm MV}} \rho (D_{\RR\times M}) = \rho (D_M)\;\;\text{in}\;\;
K_{n+1} (D^* (M)^\Gamma) .$$
In particular,
$$\delta_{{\rm MV}} \rho_\Gamma (D_{\RR\times M}) = \rho_\Gamma (D_M)\;\;\text{in}\;\; K_{n+1} 
(D^*_\Gamma).$$
From this equation and the very definition 
of $\rho^{\rm pm}_\Gamma$ class , we obtain at once that 
$$\rho^{{\rm pm}}_\Gamma (D_{\RR\times M})=
\rho_\Gamma (D_M) \;\;\text{in}\;\;  K_{n+1}(D^*_\Gamma).$$
Thus it suffices to prove that
\begin{equation}\label{reduction-rho-pm}
 \rho^{{\rm pm}}_\Gamma (D_{W})=\rho^{{\rm pm}}_\Gamma (D_{\RR\times M})
 \end{equation}
In order to show this  equality  
 we observe, first of all, that it suffices to prove that
$$\rho^{{\rm pm}}_\Gamma (D_W)= \rho^{{\rm pm}}_\Gamma (D_{\cyl,+}) \quad\text{in}\quad
K_{n+1} (D^*_\Gamma)\,.$$
Indeed, if this equality holds we can further modify $W_{\cyl,+}$ by cutting out $M_+$
and gluing in at its place $\RR_{\geq} \times M$, obtaining from $W_{\cyl,+}$ the manifold
$W_{\cyl,\cyl}$, which is nothing but $\RR\times M$. By the same argument above we 
obtain the equality
$$\rho^{{\rm pm}}_\Gamma  (D_{\cyl,+})=  \rho^{{\rm pm}}_\Gamma  (D_{\cyl,\cyl})\equiv \rho^{{\rm pm}} (D_{\RR\times M})
\quad\text{in}\quad
K_{n+1} (D^*_\Gamma)\,.$$
This proves that $$\rho^{{\rm pm}}_\Gamma  (D_W)=\rho^{{\rm pm}}_\Gamma  (D_{\RR\times M})
\quad\text{in}\quad
K_{n+1} (D^*_\Gamma)$$
which is precisely \eqref{reduction-rho-pm}. In order to show that $\rho^{{\rm pm}}_\Gamma  (D_W)= 
\rho^{{\rm pm}}_\Gamma  (D_{\cyl,+})$ we use Lemma
\ref{lem:reduction-partitioned} and the following commutative
diagram. For the sake of brevity we set

\noindent
$
A:= \relD{W_{\cyl,+}}{W_+}^\Gamma / \relD{W_{\cyl,+}}{M}^\Gamma$,
$
B:=D^* (W_+)^\Gamma / \relD{W_{+}}{M}^\Gamma$,
$
C=\relD{W}{W_+}^\Gamma / \relD{W}{M}^\Gamma
$,

\noindent
$ A_\Gamma := \relD{ \RR_{\geq}\times
  E\Gamma}{\RR\times E\Gamma}^\Gamma /
  \relD{\RR\times E\Gamma}{\{0\}\times E\Gamma }^\Gamma,
$
$B_\Gamma:=D^* (\RR_{\geq}\times E\Gamma )^\Gamma /
\relD{\RR_{\geq}\times
  E\Gamma}{\{0\}\times E\Gamma}^\Gamma $,

\noindent
$ \mathfrak{A}_\Gamma:=
\relD{\RR\times E\Gamma}{\{0\}\times E\Gamma }^\Gamma $,
$ \mathfrak{B}_\Gamma:=
\relD{\RR_{\geq}\times E\Gamma}{\{0\}\times E\Gamma }^\Gamma $.

 \noindent
We notice furthermore that $n\equiv 1\,{\rm mod} \,2$.

\begin{equation*}
\begin{CD}
K_1 (A) @>>> K_1 (A_\Gamma) @>{\pa}>> K_0 ( \mathfrak{A}_\Gamma)@>{\iso}>> K_0 (D^*_\Gamma)\\
@AA{\iso}A @AA{\iso}A  @AAA @AA{=}A\\
K_1 (B) @>>> K_1 (B_\Gamma)@>{\pa}>> K_0 ( \mathfrak{B}_\Gamma)@>{\iso}>> K_0 (D^*_\Gamma)\\
@VV{\iso}V @VV{\iso}V @VVV @VV{=}V\\
K_1 (C) @>>> K_1 (A_\Gamma) @>{\pa}>> K_0 ( \mathfrak{A}_\Gamma)@>{\iso}>> K_0 (D^*_\Gamma)\\
\end{CD}
\end{equation*}

The class $\rho^{{\rm pm}}_\Gamma  (D_W)$ can be  obtained by mapping the class 
$ [\chi_+ D_W \chi_+]\in K_1 (C)$
all the way to $K_0 (D^*_\Gamma)$ via the homomorphisms
of the bottom  horizontal line. Here the naturality
of the boundary map in K-theory has been used. This same class can also be
computed, always applying commutativity and naturality, by lifting  
 $ [\chi_+ D_W \chi_+]\in K_1 (C)$ to $K_1 (B)$ and 
 then traveling on the central horizontal line.
The same argument applies to $\rho^{{\rm pm}}_\Gamma  (D_{\cyl,+})$, which is originally
defined by considering $[\chi_+ (D_{\cyl,+}) \chi_+]$ in $K_1 (A)$ and then traveling on the top horizontal line;
the resulting class in $K_0 (D^*_\Gamma)$
can also be obtained by lifting $[\chi_+ (D_{\cyl,+}) \chi_+]$ to $K_1 (B)$
and then traveling on the central horizontal line. Since, by the lemma, the two lifts of 
$ [\chi_+ D_W \chi_+]$ and $[\chi_+ (D_{\cyl,+}) \chi_+]$ are equal in $K_1 (B)$, we see that
$\rho^{{\rm pm}}_\Gamma  (D_W)= \rho^{{\rm pm}}_\Gamma  (D_{\cyl,+})$, as required.

Summarizing, assuming Lemma \ref{lem:reduction-partitioned} and 
the cylinder delocalized index theorem \ref{theo:fundamental-cylinder},
 we have proved that
$$\rho^{{\rm pm}}_\Gamma (D_{W})=\rho^{{\rm pm}}_\Gamma (D_{\RR\times M})\equiv 
\delta_{{\rm MV}} (\rho_\Gamma (D_{\RR\times M}))= \rho_{\Gamma} (D_M)$$
and this is precisely what we need to show in order to establish Theorem \ref{theo:part_mf_abst}

\medskip
\noindent
We shall now  prove  Lemma \ref{lem:reduction-partitioned}.\\
Consider more generally the following situation: we have two complete $\Gamma$-manifolds 
as above, $W$ and $Z$,
both endowed with metrics of positive scalar curvature and with partitions
$$W=W_1 \cup_{M} W_2\,,\quad Z=W_1 \cup_M Z_2\,.$$
In other words, the two partitions have one component equal, $W_1$, they (necessarily) 
involve the same hypersurface, $M$, but have the other component of the partition different. 
Choose a chopping function $\chi$ equal to $\pm 1$ on both the spectrum
of $D_W$ and $D_Z$. We want to show that 
\begin{equation}\label{equality-lemma}
\chi_{W_1} (\chi (D_W)) \chi_{W_1} = \chi_{W_1} (\chi (D_Z)) \chi_{W_1}\;\;\text{in}\;\;
D^* (W_1)^\Gamma / \relD{W_1}{M}^\Gamma \,,
\end{equation}
with $\chi_{W_1}$ denoting the characteristic function of $W_1$.
Consider the left hand side of the above equation, $\chi_{W_1} (\chi (D_W)) \chi_{W_1}$.
The Fourier transform of $\chi$ is a smooth rapidly decreasing function
away from $0$, see \cite[p.121]{roe-partitioning}.
Thus $\chi_{W_1} (\chi (D_W)) \chi_{W_1}$ is (up to multiplication with
$\sqrt{2\pi}$) approximated in norm for $R\in\reals$ large by
$$\chi_{W_1} \left(\int_{-R}^R \hat{\chi} (\xi ) e^{i\xi D_W} d\xi \right) \chi_{W_1}.$$
We rewrite this latter term as
\begin{equation}\label{useful-sum}
\begin{split}
\chi_{(W_1\setminus U_{R} (M))} \left(\int_{-R}^R \hat{\chi} (\xi ) e^{i\xi D_W} d\xi \right)\chi_{(W_1\setminus U_{R} (M_1))}  
+ \chi_{W_1} \left(\int_{-R}^R \hat{\chi} (\xi ) e^{i\xi D_W} d\xi \right) \chi_{U_{R} (M)}\\
+ \chi_{U_{R} (M)} \left(\int_{-R}^R \hat{\chi} (\xi ) e^{i\xi D_W} d\xi \right)\chi_{(W_1\setminus U_{R} (M))}
\end{split}
\end{equation}
Because of the unit propagation property, the first summand is unchanged if
we replace $D_W$ by $D_Z$, i.e.~is equal to
$$\chi_{(W_1\setminus U_{R} (M))} 
\left(\int_{-R}^R \hat{\chi} (\xi ) e^{i\xi D_Z} d\xi
\right)\chi_{(W_1\setminus U_{R} (M))} \;.$$
We then rewrite \eqref{useful-sum} as
\begin{equation*}
\begin{split}
\chi_{W_1} \left(\int_{-R}^R \hat{\chi} (\xi ) e^{i\xi D_Z} d\xi \right)\chi_{W_1}  
+ \chi_{W_1} \left(\int_{-R}^R \hat{\chi} (\xi ) e^{i\xi D_W} d\xi \right) \chi_{U_{R} (M)}\\
+ \chi_{U_{R} (M)} \left(\int_{-R}^R \hat{\chi} (\xi ) e^{i\xi D_W} d\xi \right)\chi_{(W_1\setminus U_{R} (M))} 
-\chi_{W_1} \left(\int_{-R}^R \hat{\chi} (\xi ) e^{i\xi D_Z} d\xi \right) \chi_{U_{R} (M)} \\
- \chi_{U_{R} (M)} \left(\int_{-R}^R \hat{\chi} (\xi ) e^{i\xi D_Z} d\xi
  \right)\chi_{(W_1\setminus U_{R} (M))} \;.
\end{split}
\end{equation*}
The first summand in this sum approximates $\chi_{W_1} (\chi (D_Z))\chi_{W_1} $; moreover,
by unit propagation the remaining four summands are elements in the
ideal $\relD{W_1}{M}^\Gamma $.
Therefore the difference $$\chi_{W_1} (\chi (D_W))\chi_{W_1}-
\chi_{W_1} (\chi (D_Z))\chi_{W_1} $$ is approximated by a sequence of elements in the ideal
$\relD{W_1}{M}^\Gamma$; since this ideal is closed we have proved that 
 $\chi_{W_1} (\chi (D_W))\chi_{W_1}-
\chi_{W_1} (\chi (D_Z))\chi_{W_1} \in \relD{W_1}{M}^\Gamma$, i.e.
that  $$\chi_{W_1} (\chi (D_W))\chi_{W_1}=
\chi_{W_1} (\chi (D_Z))\chi_{W_1} \;{\rm  mod}\; \relD{W_1}{M}^\Gamma.$$ The
lemma is proved.

\smallskip
\noindent
The proof of the partitioned manifold theorem for $\rho$-classes,
Theorem \ref{theo:part_mf_abst}, is now complete.

\begin{remark}
  In \cite{Siegel} and \cite{Schick-Zadeh} the classical partitioned manifold
  index theorem is extended to a multi-partitioned situation, i.e.~to a
  manifold partitioned by $k$ suitably transversal hypersurfaces. It would be
  interesting to generalize also our $\rho$-index theorem to the
  multi-partitioned situation. This does not seem straightforward if one only
  assumes a product structure near the (codimension $k$) intersection of the
  $k$ hypersurfaces. 
\end{remark}

\subsection{Proof of  Theorem \ref{theo:fundamental-cylinder}} \label{sec:proof_cylindercase}

In this subsection, we finally prove
Theorem \ref{theo:fundamental-cylinder}, the cylinder delocalized index
theorem.\\ 
We will only treat the case in which $n+1$, the dimension of
the cylinder $\RR\times M$, is even.
We want to show  the equality 
$$\delta_{{\rm MV}} (\rho(D_{\RR\times M}))= \rho (D_M) \quad\text{in}\quad K_0 (D^* (M)^\Gamma).$$
Equivalently, 
see Remark \ref{remark:explicit_mv}, we want to show that 
$$\pa [ \psi_+ \chi(D_{\cyl})_+
\psi_+ ]=j_M\, \rho(D_M) \quad\text{in}\quad K_{0} ( \relD{[0,\infty)\times 
   M}{M}^\Gamma)$$
 where 
 $j_M  \colon K_*  (D^*(M)^\Gamma )$
$\to$ $K_{*} ( \relD{[0,\infty)\times M}{M}^\Gamma)$
is the map induced by the inclusion $M\hookrightarrow
[0,\infty)\times M$ and $\pa$ is the boundary map 
 $  K_{1} (D^* ( [0,\infty)\times M)^\Gamma / \relD{[0,\infty)\times 
 M}{M}^\Gamma) 
\rightarrow K_{0} ( \relD{[0,\infty)\times 
   M}{M}^\Gamma)$.

 
 \begin{notation}\label{not:proof-cyl}
 In order to lighten the notation we shall always write  $\mathcal{L}^2 (M)$
 for the covariant $M$-module given by the $L^2$-section of the spinor bundle
 of $M$; the latter is denoted
 $S_M$. We shall consider  $\RR\times M$ and write $\mathcal{L}^2_{\oplus}
 (\RR\times M)$
 for the $L^2$-sections of the bundle obtained by pulling back $S_M\oplus S_M$ from $M$
 to $\RR\times M$. We keep the notation $\mathcal{L}^2 (\RR\times M)$  for the $L^2$-sections of  $S_M$.
 Similar notations are adopted for $[0,\infty)\times M\equiv \RR_{\geq}\times M$, the half cylinder.
 Departing from the notation adopted so far, and only for this subsection, we denote by  $D$
  the ($\Gamma$-equivariant)  Dirac operator on $M$ and by $D_{\cyl}$
 the ($\Gamma$-equivariant) operator on $\RR\times M$ (these being the only Dirac operators we will be concerned with).
 \end{notation}
 
 We thus tackle the proof of the
  identity
\begin{equation}\label{iota-pa-2bis}
\pa [\psi_+ \chi(D_{\cyl})_+ \psi_+]=j_M\, \rho(D)  \end{equation}
 in $ K_{n+1} ( \relD{[0,\infty\times M)}{M}^\Gamma)$,
with $\psi_+$ the characteristic function of $[0,\infty)\times M$ in
$\RR\times M$. 

In order to establish \eqref{iota-pa-2bis} we need 
an explicit representative for the right hand side. To define $j_{M}$
 we must  find  an isometry  $V\colon \mathcal{L}^2 (M)\to  \mathcal{L}^2
 (\RR_{\geq}\times M)$
 covering in the $D^*$-sense the inclusion $M\hookrightarrow
\RR_{\geq}\times M$, $m\mapsto (0,m)$; see Section \ref{sec:intro}.
Of course we know that one can always find covariant modules $H_1$ for $M$,
$H_2$ for
$\RR_{\geq}\times M$ and an isometry $V\colon H_1\to H_2$ covering the inclusion
in the $D^*$-sense; 
here we want to show that we can choose $H_1= \mathcal{L}^2 (M)$,
$H_2= \mathcal{L}^2 (\RR_{\geq}\times M)$ and then describe
explicitly the isometry $V$.

Consider $ \mathcal{L}^2 (\RR_{\geq}\times M)$; this can be identified with $L^2 ([0,\infty), 
\mathcal{L}^2 (M))$.
Define $V$
as follows: 
\begin{equation}\label{V}
\mathcal{L}^2 (M)\ni s\mapsto Vs\in L^2 ([0,\infty), 
\mathcal{L}^2 (M))\,,\;\;(Vs)(t):= \sqrt{2|D|}e^{-t|D|}(s)\,.\end{equation}

 \begin{proposition}\label{prop:v-covers}
 The bounded  linear operator $V\colon \mathcal{L}^2 (M)\to  \mathcal{L}^2
 (\RR_{\geq}\times M)$
 of \eqref{V}
 covers in the $D^*$-sense the inclusion $i\colon M\hookrightarrow
 \RR_{\geq}\times M$, $m\to (0,m)$. 
 \end{proposition}
 
 \begin{proof}
 We prove this in Subsection \ref{sec:proof_of_important_props}.
 \end{proof}
 
 We now observe that by its very definition
  the $\rho$-class  of the operator on $M$,
$\rho (D)=[ \chi_{[0,\infty)} (D)]\;\;\text{in}\;\;  K_0 ( D^*(M)^\Gamma ) $.

    We have, by definition of $j_{M }$,
 $$j_{M} [ \chi_{[0,\infty)} (D)]= [V \chi_{[0,\infty)} (D) V^*]\in K_0 (
 \relD{\RR_{\geq}\times M}{M}^\Gamma)\,.$$
The operator  $P:=V \chi_{[0,\infty)} (D) V^*$, which is a projector,  acts as
follows 
on $L^2 ([0,\infty),\mathcal{L}^2 (M))$:
\begin{equation}\label{conjugated-rho}
(V\chi_{[0,\infty)} (D) V^*g)(t)=  \int_0^\infty \sqrt{2|D|}e^{-t|D|}  \chi_{[0,\infty)} (D)
\sqrt{2|D|}e^{-
\tau |D|}g d\tau  .
\end{equation}
Thus, we need to show that the K-theory class of the projector $P$ given by \eqref{conjugated-rho}
coincide with the index class  $\pa [ \psi_+ \chi(D_{\cyl})_+  \psi_+]$. In order to achieve this
it suffices to show that there exists an $L\in D^*(\RR_{\geq}\times M)$
such that
\begin{equation}\label{adapted-parametrix}
(\psi_+ \chi(D_{\cyl})_+ \psi_+) \circ L= \Id\,; \quad L\circ (\psi_+ \chi(D_{\cyl})_+ \psi_+) = \Id - P
\end{equation}
Indeed, from the very definition of the boundary map, see \cite{blackadar}, we would then have
that $\pa [ \psi_+ \chi(D_{\cyl})_+ \psi_+]=[P]$ which is what we wish to prove.

In order to find such an $L$ we first perform a deformation of the
representatives of the class 
$[ \psi_+ \chi(D_{\cyl})_+ \psi_+]$ in $K_1 (D^*
  ( \RR_{\geq}\times M)^\Gamma/ \relD{\RR_{\geq}\times M}{M}^\Gamma )$. Let us denote by
  $t$ the variable on the line $\RR$ appearing in $\RR\times M$; let us denote
  by $\pa_t$ the derivative with respect to $t$. We first concentrate our
  analysis on $\RR\times M$. 
  The operator $D_{\cyl}$ is given as
  \begin{equation}\label{d-cyl}
 \left( \begin{matrix}
      0 & D -\pa_t \\ D + \pa_t & 0
    \end{matrix}\right)\,.
  \end{equation}
  where we recall that $D$ denotes the $\Gamma$-equivariant Dirac operator on $M$.
  We have already observed that since
   $D$ is assumed to be $L^2$-invertible (the scalar curvature on $M$ is positive), also $D_{\cyl}$ is $L^2$-invertible.
 We choose as a chopping function $\chi$ the one given by 
 $\chi (t)= 1$ for $t\geq 0$, $\chi (t)=-1$ for $t<0$, which is continuous on
 the  spectrum of $D_{cyl}$.  Thus $\chi (D_{\cyl})_+$ is the bounded operator
 on $\mathcal{L}^2 (\RR\times M)$ given by the left
 bottom corner of 
  \begin{equation*}
 \left( \begin{matrix}
      0 & \frac{D -\pa_t}{\sqrt{D^2-\pa_t^2}} \\ \frac{D + \pa_t}{\sqrt{D^2-\pa_t^2}} & 0
    \end{matrix}\right)\,.
  \end{equation*}
  The operator $\chi(D_{\cyl})_+$ will be written as 
  $\frac{D + \pa_t}{\sqrt{D^2-\pa_t^2}} $. We shall connect it to 
  \begin{equation*}
  \frac{|D| + \pa_t}{D-\pa_t}
  \end{equation*}
  which is also an invertible operator on $\mathcal{L}^2 (\RR\times M)$.
  We observe here a few useful identities: $[D,\pa_t]=0=[|D|,\pa_t]$;
  $(|D|-\pa_t)(|D|+\pa_t)=D^2-\pa^2_t=(D-\pa_t)(D+\pa_t)$. The latter equality
  gives 
  $$ \frac{|D|+\pa_t}{D-\pa_t}=\frac{D+\pa_t}{|D|-\pa_t}\,.$$ Notice, in
  particular, that 
  \begin{equation}\label{inverse}
  \left( \frac{|D| + \pa_t}{D-\pa_t}\right)^{-1}=\frac{|D|-\pa_t}{D+\pa_t}\,.
  \end{equation}
  We claim that the line segment joining the two operators, 
  $$ s\mapsto s \frac{D + \pa_t}{\sqrt{D^2-\pa_t^2}} + (1-s)  \frac{|D| + \pa_t}{D-\pa_t}$$
  is through $L^2$-invertible operators in $D^* (\RR\times M)^\Gamma $.\\The fact that for each $s\in [0,1]$ the
  above operator is invertible can be seen by rewriting it as $A_s/((D-\pa_t)\sqrt{D^2-\pa_t^2})$ with 
  $A_s=s(D^2-\pa_t^2)+(1-s) (|D|+\pa_t)\sqrt{D^2-\pa_t}$; it suffices to show that $A_s$
  is invertible for each $s\in [0,1]$, which in turn is proved with an
  elementary computation 
  by showing  that $A_s^* A_s >0$ (here the $L^2$-invertibility of $D$ and
  $D_{\cyl}$ is used).

  Next we address the fact that $ s \frac{D + \pa_t}{\sqrt{D^2-\pa_t^2}} + (1-s)  \frac{|D| + \pa_t}{D-\pa_t}
  \in D^* (\RR\times M)^\Gamma $ for each $s$. Since $D^*$ is a $C^*$-subalgebra of the bounded of operators
  of $\mathcal{L}^2 (\RR\times M)$, it suffices to show that the end points of the convex combination 
  are in $D^* (\RR\times M)^\Gamma$. We already know that $\frac{D + \pa_t}{\sqrt{D^2-\pa_t^2}}$
  is in $D^* (\RR\times M)^\Gamma $ (given that is is the left bottom corner of $D_{\cyl}/|D_{\cyl}|$). Thus we only need to establish the following 
  
 \begin{proposition}\label{prop:deformed-in-d}
 The operator 
  $\frac{|D| + \pa_t}{D-\pa_t}$ belongs to 
 $ D^* (\RR\times M)^\Gamma$.
 \end{proposition}
 \begin{proof}
We prove this in Subsection \ref{sec:proof_of_important_props}.
 \end{proof}
  
Using the  above convex combination we see easily  that 
\begin{equation*}
[\psi_+  \frac{D + \pa_t}{\sqrt{D^2-\pa_t^2}} \psi_+] =[\psi_+  \frac{|D| + \pa_t}{D-\pa_t} \psi_+]\;\;\text{in}\;\;
K_1 (D^* (\RR_{\geq}\times M)^\Gamma /\relD{\RR_{\geq}\times M}{M}^\Gamma )
\end{equation*}
In particular $\pa [\psi_+  \frac{D+ \pa_t}{\sqrt{D^2-\pa_t^2}} \psi_+] =\pa [\psi_+  \frac{|D| + \pa_t}{D-\pa_t} \psi_+]$
in $K_0 (\relD{\RR_{\geq}\times M}{M}^\Gamma )$ 
and so we are reduced to the problem of finding
$Q\in D^*(\reals_{\ge}\times M)^\Gamma$ such that 
\begin{equation}\label{adapted-parametrix-right}
(\psi_+   \frac{|D| + \pa_t}{D-\pa_t} \psi_+) \circ Q= \Id_{\mathcal{L}^2
  (\RR_{\geq}\times M)},
\end{equation}
\begin{equation}\label{adapted-parametrix-left}
 Q\circ (\psi_+ \frac{|D| + \pa_t}{D-\pa_t} \psi_+) = \Id_{\mathcal{L}^2
   (\RR_{\geq}\times M)} - P \;.
 \end{equation}
We claim that we can take $Q:= \psi_+ (\frac{|D|-\pa_t}{D+\pa_t}) \psi_+$. First, by
Proposition \ref{prop:deformed-in-d}, $Q\in D^*(\reals_{\ge}\times M)^\Gamma $.
Next we need to show that with this choice
\eqref{adapted-parametrix-right} holds.
Using that $(\psi_+)^2=\psi_+$, together with  \eqref{inverse} we see that it suffices to show that
$\psi_+(\frac{|D|-\pa_t}{D+\pa_t}) \psi_+ = (\frac{|D|-\pa_t}{D+\pa_t})
\psi_+$ on $\mathcal{L}^2 (\RR\times M)$.
We decompose $\mathcal{L}^2 (\RR\times M)$ using the spectral transform
induced by  the Browder-Garding decomposition associated to the self-adjoint
operator $D$, see \cite{ramachandran}. Thus there is an isometry 
\begin{equation}\label{spectral_transform}
T\colon \mathcal{L}^2 (\RR\times M)\to L^2 (\RR,\oplus_{j} L^2 (\RR,d\mu_j) )
\end{equation}
 such that 
$\psi_+(\frac{|D|-\pa_t}{D+\pa_t}) \psi_+ = (\frac{|D|-\pa_t}{D+\pa_t}) \psi_+$ if and only if
$$\chi_{[0,\infty)} \frac{|\lambda |-\pa_t}{\lambda+\pa_t}\chi_{[0,\infty)}= \frac{|\lambda |-\pa_t}{\lambda+\pa_t}\chi_{[0,\infty)}$$ on each $L^2 (\RR,L^2 (\RR,d\mu_j))$. Conjugate both sides of the last equation by 
Fourier transform for the $t$-variable $\mathcal{F}$, then write
$\mathcal{F}^{-1} \frac{|\lambda |-\pa_t}{\lambda+\pa_t}\chi_{[0,\infty)} 
\mathcal{F}$ as 
($\mathcal{F}^{-1} \frac{|\lambda |-\pa_t}{\lambda+\pa_t}\mathcal{F})( \mathcal{F}^{-1}\chi_{[0,\infty)}
\mathcal{F}$). Then the right hand side
is equal to the projection onto the Hardy space, see \cite{stein-shakarchi},
followed by the multiplication operator
by $\frac{|\lambda
  |+i\tau}{\lambda-i\tau}=\frac{\lambda+i\tau}{\abs{\lambda}-i\tau}$. The
latter function is holomorphic on the upper 
half plane, so multiplication by it does preserve the Hardy space. It follows
that projecting once again  onto
the Hardy space leaves  it unchanged. This proves that
$\chi_{[0,\infty)} \frac{|\lambda |-\pa_t}{\lambda+\pa_t}\chi_{[0,\infty)}= \frac{|\lambda |-\pa_t}{\lambda+\pa_t}\chi_{[0,\infty)}$ and thus that  \eqref{adapted-parametrix-right}  holds.
Next we tackle \eqref{adapted-parametrix-left}, i.e. the equation 
\begin{equation}\label{adapted-parametrix-left-bis}
 \psi_+ (\frac{|D|-\pa_t}{D+\pa_t}) \psi_+\circ (\psi_+ \frac{|D| +
   \pa_t}{D-\pa_t} \psi_+) = \Id_{\mathcal{L}^2 (\RR_{\geq}\times M)} - P
 \end{equation}
with    $P\colon L^2 ([0,\infty), 
\mathcal{L}^2 (M))\to L^2 ([0,\infty), 
\mathcal{L}^2 (M))$ 
defined by 
\begin{equation}\label{eq:def_of_P}
  (Pg)(t)=  \int_0^\infty \sqrt{2|D|}e^{-t|D|}  \chi_{[0,\infty)} (D_{\pa})
\sqrt{2|D|}e^{-u |D|}g du \;.
\end{equation}

\begin{lemma}\label{lem:proj-lambda}
Under the spectral transform 
\begin{equation}\label{spectral-tr}
T\colon \mathcal{L}^2 ([0,\infty)\times M)\to L^2 ([0,\infty),\oplus_{j} L^2 (\RR,d\mu_j) )
\end{equation}
the projector $P$ diagonalizes. Let $P_j$ be the restriction of $P$ to
$L^2 ([0,\infty), L^2 (\RR_{\lambda},d\mu_j) )$ and consider the decomposition 
$$L^2 ([0,\infty), L^2 (\RR,d\mu) )=L^2 ([0,\infty), L^2 ([0,\infty)_\lambda,d\mu))\oplus 
L^2 ([0,\infty), L^2 ((-\infty),0)_\lambda,d\mu))\,.$$
Then $P_j$ is diagonal and equal to 
$$\left( \begin{matrix}
      P^+_j & 0 \\ 0 & 0 
    \end{matrix}\right)$$
    with $P^+_j$  described in the following way:\\
    view $L^2 ([0,\infty), L^2([0,\infty)_\lambda),d\mu_j)$ as $L^2_{d\mu_j} ([0,\infty)_\lambda), L^2 [0,\infty)_t)$. With respect to this decomposition, $P^+_j$ is the direct integral $\int_{[0,\infty)} P_\lambda\, d\mu_j(\lambda)$ where $P_\lambda\colon L^2([0,\infty))\to L^2([0,\infty))$ is the projector onto the subspace spanned by $f_\lambda(t)=\sqrt{2\lambda}e^{-\lambda t}$.
   
    Observe that, because $0$ is not in the spectrum of the operator $D$, $0$
    is not in the support of any of the measures $\mu_j$ so that $f_\lambda$
    indeed is in $L^2([0,\infty))$ and depends continuously on $\lambda$
    for all $\lambda$ relevant to us.
\end{lemma}

\begin{proof}
  With respect to the decomposition of our Hilbert space, as $D$ acts as multiplication with $\lambda$ and $|D|$ acts as multiplication with $|\lambda|$ under the spectral transform $T$, the formula \eqref{eq:def_of_P} becomes the direct integral over the operators 
  $$g\mapsto \int_0^\infty \sqrt{2|\lambda|}e^{-t|\lambda|} \chi_{[0,\infty)}(\lambda) \sqrt{2|\lambda|}e^{-u|\lambda|} g(u)\,du= \begin{cases} f_\lambda \langle g,f_\lambda\rangle_{L^2([0,\infty))}; & \lambda>0 \\ 0 ; & \lambda<0 \end{cases}. $$
\end{proof}

We are now in the position to prove
\eqref{adapted-parametrix-left-bis}. We use the spectral transform
and we reduce to a computation on each single $L^2 ([0,\infty), L^2 (\RR,d\mu)
)$. 
First we remark that the operator induced by the left hand side of \eqref{adapted-parametrix-left-bis}
diagonalizes with respect to 
the decomposition
\begin{equation*}
  L^2 ([0,\infty), L^2 ([0,\infty)_\lambda,d\mu))\oplus L^2 ([0,\infty), L^2
  ((-\infty),0)_\lambda,d\mu),
\end{equation*}
even better, as before it becomes a direct integral over $\lambda\in \RR$ with measure $\mu_j(\lambda)$.
 Moreover,  the restriction  to
the second summand is equal to the identity, given that the restriction of the operator induced by 
$\frac{|D| + \pa_t}{D-\pa_t}$ on $L^2 (\RR, L^2 (\RR,d\mu))$ is equal to $(|\lambda| + \pa_t)(\lambda-\pa_t)^{-1}$
and it is therefore equal to $-\Id$ on $L^2 (\RR, L^2 (-\infty,0)_\lambda,d\mu)$, and the same holds with the same argument for $\frac{|D|-\pa_t}{D+\pa_t}$. The conclusion is that the restriction
of the two sides of \eqref{adapted-parametrix-left-bis} to the second summand of $L^2 ([0,\infty), L^2 ([0,\infty)_\lambda,d\mu))\oplus 
L^2 ([0,\infty), L^2 ((-\infty),0)_\lambda,d\mu))$ agree.

 Thus we are left with the task of showing that the
operator induced by the left hand side of \eqref{adapted-parametrix-left-bis} on $L^2 ([0,\infty), L^2 ([0,\infty)_\lambda,d\mu))$ in its direct integral decomposition for each $\lambda\in (0,\infty)$ has a 1-dimensional null space, generated by $f_\lambda(t)= \sqrt{2\lambda}e^{-t\lambda}$, $t\geq 0$
and it is equal to the identity on the orthogonal complement of this null space.
Using the direct integral decomposition, we treat $\lambda>0$ as a constant. Let us then check that the function $f_\lambda(t)$ is in the null space of the
operator induced by the left hand side of \eqref{adapted-parametrix-left-bis}. In order to check this
property we conjugate by Fourier transform. The inverse Fourier transform of $\sqrt{2\lambda}e^{-\lambda t}$, $t\geq 0$,
is up to a constant equal to $1/(\lambda-i\tau)$ (compare
\cite[Appendix]{Kammler}) which is holomorphic outside $\tau=-i\lambda$; in
particular it
is holomorphic on the upper half plane which means that it is left unchanged by the projection
onto the Hardy space (as it should, given that $\chi_{[0,\infty)}f_\lambda=f_\lambda$). Now we apply the operator
of multiplication by $(\lambda-i\tau)(\lambda+i\tau)^{-1}$, getting the function $1/(\lambda+i\tau)$. This is holomorphic on the lower half plane and therefore its boundary value is in the orthogonal complement of the Hardy space, so it is mapped to $0$ by projecting onto the Hardy space. The conclusion is that
 $f_\lambda$ is indeed in the null space of the left hand side  of 
 \eqref{adapted-parametrix-left-bis}. 
 
 Consider now the
 orthogonal complement of $f_\lambda$, i.e. 
 $$\{g\in L^2 ([0,\infty)) ; \int_0^\infty g(t) f_\lambda(t)\,dt=\int_\RR g(t)e^{i t (i\lambda)}\;dt=0\}.$$
 This is the space of functions $g$ such that $(\mathcal{F}^{-1}{g})(i\lambda )=0$. This means that 
 $$ \tau\mapsto \frac{(\lambda-i\tau)}{(\lambda+i\tau)} \left(\mathcal{F}^{-1}{g}\right)(\tau)$$
 is still holomorphic on the upper half plane and thus projection onto the Hardy space leaves
 it unchanged. Composing with the multiplication operator by $\frac{(\lambda+i\tau)}{(\lambda-i\tau)}$
 gives back $\mathcal{F}^{-1}{g}$. Thus the left hand side of 
 \eqref{adapted-parametrix-left-bis} acts as the identity on the orthogonal complement
 of $f_\lambda$ and the conclusion is that the left hand side of 
 \eqref{adapted-parametrix-left-bis} is precisely equal to $\Id-P$, thanks to Lemma
 \ref{lem:proj-lambda}.
 
 \subsection{Proof of  Propositions \ref{prop:v-covers} and
   \ref{prop:deformed-in-d}} 
\label{sec:proof_of_important_props}
We begin by proving  Proposition \ref{prop:v-covers}.\\
We wish to prove that the bounded  linear operator $V\colon \mathcal{L}^2
(M)\to  \mathcal{L}^2 (\RR_{\geq}\times M)$
 defined by $(Vs)(t):= \sqrt{2|D|}e^{-t|D|}(s)$
 is an isometry, that it satisfies the property that $\phi V - V (\phi\circ i)$ is
\emph{compact} for each  $\phi\in C_0 (\RR_{\geq}\times  M)$, and finally that
it is the norm-limit
    of bounded linear operators  $U$ satisfying the  propagation condition 
    appearing in Definition \ref{def:cover-D*-sense}.\\
    The fact that $V$ is an isometry is proved by direct computation, using
    the fact that the spectrum of
    $D$ does not contain zero, so that $e^{-t |D|}$ converges (exponentially)
    to zero for $t\to +\infty$.
Consider next the propagation condition which we recall here:
{\it there exists an $R>0$  such that $\phi U \psi=0$ if $d({\rm supp} \phi, i({\rm supp }\psi))>R$, with 
$\phi\in C_0 (\RR_{\geq}\times M)$ and $\psi\in C_0 (M)$}. We must find an approximating
sequence of bounded linear operators $U$ with this property. 
Consider the function $h_t (x):= \sqrt{|x|}e^{-t|x|}$. Our operator $V$ is
obtained from $h_t$ by
$$(V s) (t) = \frac{\sqrt{2}}{\sqrt{2\pi}} \int_{\RR} \widehat{h_t} (\xi) e^{i\xi D} (s) d\xi\,.$$
 We consider the function $f_t (x):=
 \begin{cases}
   0; & x<0\\ \sqrt{x}e^{-t x}; & x\geq 0
 \end{cases}
$. We write $h_t = f_t + g_t$, with $g_t (x):= f_t (-x)$.
Its Fourier-Laplace transform is
$$\widehat{f_t} (\zeta)= \frac{1}{\sqrt{2\pi}} \int_{0}^{+\infty} \sqrt{x} e^{-x (t+ i\zeta)}dx \,,$$
and we observe that this is a a holomorphic function in the region $\Im
(\zeta) < t$. 
For $s<t$ the integral can easily be evaluated, giving
 $$  \widehat{f_t} (is)= C \frac{1}{(t-s)^{\frac{3}{2}}}\,,\quad\text{with}\quad C= \frac{1}{\sqrt{2\pi}}\int_0^{+\infty}
 \sqrt{x}e^{-x} dx\,.$$
 Thus, by the identity principle for holomorphic functions, we deduce that 
  $$\widehat{f_t} (\zeta)= C (t+i\zeta)^{-\frac{3}{2}}\,,$$
 with the branch of the square root such that $t^{\frac{3}{2}}$ is
 positive for $t>0$. Going back to $h_t$ we have therefore proved that for
 $\xi\in\RR$  
 $$C^{-1}\widehat{h_t} (\xi) =  (t+i\xi)^{-\frac{3}{2}} + (t-i\xi)^{-\frac{3}{2}}.$$
 Let $R\in\RR$, $R>0$. We 
 define a bounded linear operator $U_R\colon \mathcal{L}^2 (M)\to
 \mathcal{L}^2 (\RR_{\geq}\times M)$
as follows:
$$(U_R s) (t) =
\begin{cases}
  \frac{\sqrt{2}}{\sqrt{2\pi}} \int_{-R^4}^{R^4} \widehat{h_t} (\xi) e^{i\xi
    D} dx &  \text{if} \quad t\leq R \\
   0 & \text{if} \quad t>R
\end{cases}.
$$ 
It is clear that $U_R$ satisfies the propagation condition. We have
$$\| (U_R - V)(s) \|^2\leq \sqrt{2}\| e^{-R |D|} \|^2 \| s \|^2 + \sqrt{2}\int_0^R \| \int_{|\xi|> R^4} \widehat{h_t} (\xi) e^{i\xi D} d\xi (s) \|^2 dt\,.
$$
For the second summand on the right hand side we can use the explicit description of 
$\widehat{h_t} (\xi)$, to get the estimates:
\begin{equation*}
\begin{aligned}
\int_0^R \| \int_{|\xi|> R^4} \widehat{h_t} (\xi) e^{i\xi D} (s)\|^2 dt & \leq \| s \|^2 \int_0^R ( \int_{|\xi|> R^4} | \widehat{h_t} (\xi)| d\xi )^2 dt\\& \leq  \| s \|^2 \int_0^R ( \int_{|\xi|> R^4} ( C |t+i\xi|^{-\frac{3}{2}} + C |t-i\xi|^{-\frac{3}{2}}) d\xi )^2 dt\\
& \leq 2C \| s \|^2 \int_0^R (\int_{|\xi|> R^4}  |\xi|^{-\frac{3}{2}} d\xi )^2\\&\leq C^\prime \frac{1}{R} \|s\|^2.
\end{aligned}
\end{equation*}
Summarizing, we have shown that there exists a positive $C>0$ such that
$$\| (U_R - V)(s) \|^2\leq C \left(\| e^{-R |D|} \|^2 + R^{-1}\right) 
\cdot \| s \|^2$$
proving that $U_R \to V$ in operator norm as $R\to +\infty$.

Next we need to show that 
$\phi V - V (\phi\circ i)$
\emph{is compact for each} $\phi\in C_0 (\RR_{\geq}\times  M)$. 
By the Kasparov Lemma \ref{lem:kasparov}  it suffices to prove that 
$\phi_1 V \psi_2$ is compact whenever $\phi_1\in 
C_0 ([0,\infty)\times M)$ $\psi_2\in C_0 (M)$ and the image of the support of $\psi_2$
through the inclusion map $i$  is disjoint from the support of $\phi_1$. 
Clearly, it suffices to prove that $\phi_1 V \psi_2$ is the  norm limit of compact operators.
We can obviously
consider $\phi_1=\alpha\otimes\psi_1$  with $\psi_1\in C_0 (M)$  and $\alpha\in C_0 [0,\infty)$.
There are then two cases:
\begin{enumerate}
\item \label{item:alpha_away_from_0} $\alpha(0)=0$, and we may as well assume
  that $\alpha$ is   supported away from $t=0$;
\item \label{item:alpha_at_0}
 $\alpha$ is not
  supported away from $t=0$ but $d({\rm supp}\psi_1,{\rm supp} \psi_2)\geq \delta >0$.
\end{enumerate}

Let us treat \ref{item:alpha_away_from_0} first. Take $\Lambda>>0$ and
consider $\chi_{[-\Lambda,\Lambda]}$. Then we can consider
$(\alpha\otimes\psi_1) U_{\Lambda}\psi_2 $ with  $(U_{\Lambda} s)(t):= \sqrt{2 |D|} e^{-t |D| } \chi_{[-\Lambda,\Lambda]}(D) (s)$.
The operator $(\alpha\otimes\psi_1) U_{\Lambda}\psi_2 $ is compact; 
indeed  $\chi_{[-\Lambda,\Lambda]}(D)$ is in $C^* (M)$, so that 
$\chi_{[-\Lambda,\Lambda]}(D)\psi$ is compact for each $\psi\in C_0 (M)$. We are 
considering $(\alpha\otimes \psi_1) V \chi_{[-\Lambda,\Lambda]}(D)\psi_2$; since 
$\chi_{[-\Lambda,\Lambda]}(D)\psi_2$ is compact and since $V$ is an isometry (and the composition of a compact operator with a bounded operator is again compact), we conclude that $(\alpha\otimes\psi_1) U_{\Lambda}\psi_2 $ 
is compact. It remains to show that the operator norm of $
(\alpha\otimes\psi_1)V\psi_2 -(\psi_1\otimes \alpha)U_{\Lambda}\psi_2$ is small.
We shall achieve this by proving that $\alpha V - \alpha U_\Lambda$ is small
in norm. Consider the spectral transform \eqref{spectral-tr}; under this
transformation, which is an isometry,
the  operators  $V$ and $U_\Lambda$ diagonalize as the direct sum of
bounded operators 
$V_j \colon L^2 (\RR_\lambda,d\mu_j)\to L^2 [0,\infty),L^2 (\RR_\lambda,d\mu_j))$ and similarly for $U_\Lambda$.
We have: $(V_j \sigma)(t,\lambda)= e^{-t|\lambda |} \sqrt{2|\lambda |} \sigma (\lambda)$ and 
similarly for $U_{\Lambda,j}$.
 Recall that we are under the assumption that the $L^2$-spectrum of $D$ does
 not contain $0$. Thus $0$ is never in the support of any of the measures
 $\mu_j$.  
We are also under the assumption 
that $\alpha $ is supported away from $0$.
Using this and some
elementary computation
one proves that $\| \alpha V_j \psi-  \alpha U_{\Lambda,j} \|$ is
exponentially decreasing in $\Lambda$. Thus $\alpha
U_{\Lambda,j}\xrightarrow{\Lambda\to \infty}
\alpha V_j$ in norm and therefore $(\alpha\otimes \psi_1) U_{\Lambda}\psi_2
 \xrightarrow{\Lambda\to\infty} (\alpha\otimes\psi_1)V\psi_2$ in norm, which
 is what we wanted to show.\\
Next we tackle \ref{item:alpha_at_0}. It suffices to work under the
assumption that $\alpha\equiv 1$, so that we are looking at 
$\psi_1 V \psi_2$ with $d({\rm supp}\psi_1,{\rm supp} \psi_2)\geq \delta >0$. 
Let $P\colon\mathcal{L}^2 (\RR_{\geq}\times M)\to \mathcal{L}^2 (\RR_{\geq}\times M)$
 be the operator of multiplication by the characteristic function
 of the
 $t$-interval $[0,\epsilon\delta)$. Then 
$\psi_1 V \psi_2=P\psi_1 V \psi_2 + (\Id-P)\psi_1 V \psi_2$. The second summand on the right hand side
 is compact by the same argument we have employed for
 \ref{item:alpha_away_from_0}. Thus it suffices to show that the norm 
of $P\psi_1 V \psi_2 $, as an operator from $\mathcal{L}^2 (M)\to \mathcal{L}^2 (\RR_{\geq}\times M)$,
is less than $\epsilon$. Write $\psi_1 V \psi_2$ as 
$$\psi_1 (\frac{1}{\sqrt{2\pi}}\int_{\RR} \hat{h}_t (\omega)  e^{i\omega D} d\omega)\psi_2\,$$
with $h_t (\lambda):=\sqrt{2|\lambda |} e^{-t |\lambda |}$.
From the assumption $d({\rm supp}\psi_1,{\rm supp} \psi_2)\geq \delta >0$ and
the propagation $\omega$
of $e^{i\omega D}$ this is equal to 
$$\psi_1 (\frac{1}{\sqrt{2\pi}} \int_{| \omega |\geq \delta} \hat{h}_t
(\omega)  e^{i\omega D} d\omega)\psi_2\, .$$
Fix $s\in 
\mathcal{L}^2 (M)$. Then 
$$\| P \psi_1 V \psi_2 (s)\|^2_{ \mathcal{L}^2
  (\RR_{\geq}\times M)} \le \int_0^{\epsilon\delta}(\frac{1}{\sqrt{2\pi}} \int_{| \omega |\geq \delta} | \hat{h}_t (\omega)| d\omega)^2
 \|\psi_1 \|^2_{\infty}  \|\psi_2 \|^2_{\infty} \|s \|^2_{\mathcal{L}^2}dt \,.$$
 Ir remains to show that $ 
 \int_0^{\epsilon\delta} \left(\frac{1}{\sqrt{2\pi}} \int_{| \omega |\geq \delta} | \hat{h}_t (\omega)| d\omega\right)^2 dt$
 is small.  However, from the explicit computation of $ \hat{h}_t$
 we see that this is less than
 $$C \int_0^{\epsilon\delta} \left(\int_{| \omega |\geq \delta}
   |\omega|^{-3/2} d\omega\right)^2 dt\quad\text{with}
 \quad C>0.$$
 Because this latter term is equal to $C \epsilon \delta \frac{16}{\delta} =16 C\epsilon$,
 the proof of  Proposition \ref{prop:v-covers}
 is complete.\\

 We now prove Proposition  \ref{prop:deformed-in-d}. We want to show that
 $\frac{|D| + \pa_t}{D-\pa_t}$ belongs to 
 $ D^* (\RR\times M)^\Gamma $. We must prove that this operator is a norm limit of operators that
 are pseudo-local and of finite propagation.


We write 
 $\frac{|D| + \pa_t}{D-\pa_t}$ as $\frac{|D|}{D-\pa_t} +
 \frac{\pa_t}{D-\pa_t}$ and deal with the two summands separately. 

We think of $(D-\partial_t)^{-1}\colon L^2\to H^1$ as bounded operator from
$L^2$ to the Sobolev space $H^1$, and we will compose it  with  $\partial_t\colon H^1\to L^2$
or $\abs{D}\colon H^1\to L^2$ as bounded operator from $H^1$ to $L^2$.

We will show that that for $\mathcal{D}:= 
 \begin{pmatrix}
   0 & D+\pa_t \\ D-\pa_t & 0
 \end{pmatrix}$ the operator $\mathcal{D}^{-1}\colon L^2\to H^1$ can be
approximated 
 by operators $F_\epsilon$ of finite propagation in the operator norm from
 $L^2$ to 
 the Sobolov space $H^1$, such that the commutator $[F_\epsilon,\phi]\colon
 L^2\to H^1$ is compact as an operator from $L^2$ to $H^1$ whenever $\phi$ is
 (multiplication by) a compactly supported continuous function. The
 same is then true for its corner $(D-\partial_t)^{-1}$.

Secondly, we will show that also 
 $\partial_t\colon H^1\to L^2$ and $|D|\colon H^1\to L^2$
 can be approximated as
 operators from $H^1$ to $L^2$ by finite propagation operators such that the
 commmutator of the approximating operators with compactly supported functions
 is compact as operator from $H^1$ to $L^2$.

Once having achieved this, the compositions will have the same required property.

Of the three operators to study, $\partial_t$ itself has propagation zero, and
the commutator $[\partial_t,\phi]$ is multiplication with the compactly
supported function $\partial_t\phi$,
which as operator from $H^1$ to $L^2$ is compact by the Rellich lemma.

Next we study $\mathcal{D}^{-1}$. As $\mathcal{D}$ is an invertible elliptic
operator of first
   order, we can and will choose on $H^1$ the norm such that
   $\mathcal{D}\colon H^1\to L^2$ is an isometry.  
   Let $f$ be an odd smooth bounded function
 equal to $1/x$ on the spectrum of the invertible self-adjoint operator
 $\mathcal{D}$.
 Notice that $f$, and thus its Fourier transform $\hat{f}$, lies in
 $L^2(\reals)$.  We consider the 
 function $x\mapsto xf(x)$ and we arrange that $g(x):=x f(x)-1$ is compactly
 supported. Thus its Fourier 
 transform $\hat{g}(\xi)$ will be in the Schwartz space
 $\mathcal{S}$.
 This means that $\pa_{\xi} \hat{f} - \frac{1}{\sqrt{2\pi}}\delta_0$
 is an element in $\mathcal{S}$; we deduce from this that $\hat{f}$ is
 bounded, smooth outside $0$, odd and of Schwartz
 class  as $| \xi |\to +\infty$.  Write $\hat{f}=g_\epsilon + v_\epsilon$ with
 $g_\epsilon $ odd and compactly  supported  
 and $v_\epsilon\in\mathcal{S}$ with the property that $| \pa_x
 {v}_\epsilon|_{L^1(\reals)}<\epsilon$. We achieve this by setting
 $v_\epsilon:=\phi_\epsilon \hat f$ with an even smooth cutoff
 function $\phi$ which vanishes in a sufficiently large neighborhood of $0$
 and which has uniformly small derivative (using that $\hat f$ is of
 Schwartz class at $\pm \infty$). Then $f=h_\epsilon + w_\epsilon$ 
 with $h_\epsilon$ having compactly supported Fourier transform $g_\epsilon$
 and $w_\epsilon\in \mathcal{S}$ with the
 property that $|x w_\epsilon(x)|_{\infty}<2\pi \epsilon$. We deduce that 
$ f(\mathcal{D}) = \mathcal{D}^{-1}$ and
\begin{equation*}
  \norm{f(\mathcal{D})-h_\epsilon(\mathcal{D})}_{L^2\to
    H^1}=\norm{w_\epsilon(\mathcal{D})}_{L^2\to H^1} 
=\norm{\mathcal{D}w_\epsilon(\mathcal{D})}_{L^2\to L^2} \le
  \abs{x w_\epsilon(x)}_{\infty} <2\pi\epsilon .
\end{equation*}
 Finally, $h_\epsilon (\mathcal{D})$ is of finite propagation by unit
propagation speed for the Dirac type operator $\mathcal{D}$.

Let now $\phi$ be a compactly supported smooth function on $\reals\times
M$. Bearing in mind
the finite propagation (say $R_\epsilon$) of $h_\epsilon(\mathcal{D})$, we 
choose a compactly 
supported function $\psi$ which takes the value $1$ on the $R_\epsilon$-neighborhood of
the support of $\phi$. Then
$h_\epsilon(\mathcal{D})\phi=\psi h_\epsilon(\mathcal{D})\phi$ 
and $\phi h_\epsilon(\mathcal{D})=\phi h_\epsilon(\mathcal{D})\psi$. Choose a
compact spin 
manifold $X$ with an open subset $U$ which is isometric (preserving the spin
structure) to an open
neighborhood $V$ of the support of $\psi$. To construct $X$, take e.g.~the
double of a compact $0$-codimensional submanifold with boundary of $\reals\times M$ containing
the support of $\psi$.
Then (again by unit propagation speed for Dirac operators) the operators $\phi 
h_\epsilon(\mathcal{D})\psi$ and $\psi h_\epsilon(\mathcal{D})\phi$ are
unitarily equivalent to 
the corresponding operators on the compact manifold $X$; we see in this way that 
the
commutator $[\phi,h_\epsilon(\mathcal{D})]$ is unitarily
equivalent to   $[\phi_X,h_\epsilon(D_X)]\colon H^1(X)\to L^2(X)$. Here $\phi_X$ is
the function $\phi$ transported to $X$ via the isometry, and $D_X$ is the
Dirac operator on $X$. Now it is a classical fact that $h_\epsilon(D_X)$ is a
pseudodifferential operator of order $1$. This follows e.g.~from \cite[Theorem
XII.1.3]{Taylor}. 
Strictly speaking, we write $h_\epsilon(D_X)= D_X
w_\epsilon(\sqrt{D_X^2})$ which is possible because we made sure that
$h_\epsilon(x)$ is an odd function. Our original $f(x)$ is smooth and equal to
$1/x$ for $x$ large, hence is a symbol of order $-1$. More precisely, it belongs to
$S^{-1}_{1,0}(\reals)$ in the sense of \cite[Lemma
XII.1.2]{Taylor}. Now $h_\epsilon(x)$ differs from $f(x)$ by the Fourier transform
of a Schwartz function, i.e.~by a Schwartz function, i.e.~also belongs to
$S^{-1}_{1,0}$. As $h_\epsilon$ is odd, $w_\epsilon$ (satisfying
$x w_\epsilon(\abs{x})=h_\epsilon(x)$) is smooth and 
belongs to $S^{-2}_{1,0}$. By Seeley's theorem on complex powers of elliptic
operators (or the special proof given in \cite[Section XII.1]{Taylor})
$\sqrt{D_X^2}$ is a positive pseudodifferential 
operator of order $1$ with scalar valued principal symbol on the compact manifold $X$.
Thus, all the hypotheses of \cite[Theorem XII.1.3]{Taylor} are fulfilled
and we conclude that $w_\epsilon(\sqrt{D_X^2})$ is a pseudodifferential operator of order
$-2$ and $h_\epsilon(D_X)=D_Xw_\epsilon(\sqrt{D_X^2})$ is a pseudodifferential
operator of order $-1$. By standard results of the pseudodifferential
calculus this implies that its
commutator with the smooth function $\phi_X$ is a pseudodifferential operator
of order $-2$ (this is a direct consequence of the short exact sequence defined
by the principal symbol and the formula for the principal symbol of a composition). 
So, up to unitary equivalence,
$[h_\epsilon(\mathcal{D}),\phi]$
can be written as composition of the bounded operators
$[h_\epsilon(D_X),\phi_X]\colon L^2\to H^2$ and $i\colon H^2\to H^1$ where the
latter operator is compact by the Rellich lemma on the compact manifold $X$,
Therefore $[h_\epsilon(\mathcal{D}),\phi]$ indeed is compact, as we had to show. Then
also the commutators with arbitrary continuous compactly supported
functions are compact because the smooth functions are dense in
sup-norm in $C_0$.

Finally, we treat $\abs{D}\colon H^1\to L^2$. Note that this should really be
written as
$\id_{L^2 (\reals)}\tensor \abs{D}$, which is
    not a function of the Dirac 
  operator on $\reals\times M$; this makes the analysis slightly more
  complicated.

We begin by analyzing the operator $\abs{D}$ acting on $M$. We keep
considering $\abs{D}$ as a bounded operator from $H^1$ to $L^2$.
As above, we can
write $\abs{D}=k_\epsilon(D)+u_\epsilon(D)$ where $u_\epsilon$ now is an even
Schwartz 
function such that $\norm{u_\epsilon(D)}<\epsilon$ (even when considered as
operator $L^2\to L^2$) and such that $k_\epsilon(D)$ has finite propagation, say
$R_\epsilon$. Then, as above, for a smooth compactly supported function $\phi_1$
on $M$, the commutator $[k_\epsilon(D),\phi_1]$ is unitarily equivalent to
$[k_\epsilon(D_X),\phi_{1,X}]$ for a compact manifolfd $X$. And, exactly with the
same reasoning as above, $k_\epsilon(x)$ is an even symbol of order $1$, so
that $k_\epsilon(D_X)$ is a pseudodifferential operator of order $1$.
Therefore $[k_\epsilon(D_X),\phi_{1,X}]$ is a pseudodifferential operator of
order $0$, defining a bounded operator $L^2\to L^2$. So the same is
true for $[k_\epsilon(D),\phi_1]$ (which is additionally supported on a
compact subset of $M$).

Now we return to $\reals\times M$. Note that the precise meaning of
$\id_{L^2(\reals)}\otimes D\colon H^1(\reals\times M)\to L^2(\reals\times M)$
is the composition of the (bounded) embedding $H^1(\reals\times M)\into
L^2(\reals)\otimes H^1(M)$ with the bounded operator $\id_{L^2(\reals)}\otimes
D\colon L^2(\reals)\otimes H^1(M)\to L^2(\reals\times M)$. We will use this
notation throughout. We write
$$\id_{L^2 (\reals)}\tensor \abs{D}= \id_{L^2 (\reals)}\tensor k_\epsilon(D) +
\id_{L^2 (\reals)}\tensor u_\epsilon(D) \colon H^1(\reals\times M)\to
L^2(\reals\times M).
$$
The first summand on the right hand side, a bounded operator $H^1\to L^2$, has finite propagation
whereas the second has  small norm as on operator from $H^1\to L^2$.
Thus we are left with the task of proving that $\id_{L^2 (\reals)}\tensor k_\epsilon(D)$
is pseudolocal as an operator from $H^1$ to $L^2$.
Given compactly
  supported smooth function $\phi_2$ on $\reals$ and $\phi_1$ on $M$, the
  commutator $[\id_{L^2 (\reals)}\tensor k_\epsilon(D) , \phi_2\phi_1]$  equals
  $\phi_2\tensor [k_\epsilon(D),\phi_1]$ which factors as the inclusion
  $H^1\to L^2$ 
  composed with the bounded operator $\phi_2\tensor [k_\epsilon (D),\phi_1]
  \colon L^2\to L^2$.
   As, in addition, this
  commutator is compactly supported, the Rellich lemma implies that this
  composition  is compact as an operator from $H^1$ to $L^2$.
  As smooth compactly supported functions of the form $\phi_1\tensor \phi_2$
  are dense in all 
  continuous functions of compact support, this finishes the proof of
  Proposition  \ref{prop:deformed-in-d}.

\begin{remark}
  We use the calculus of pseudodifferential operators here just for
  convenience. In \cite{PiazzaSchick_surgery} we generalize the assertions to
  perturbations
  of Dirac type operators which are not necessarily pseudodifferential,
  replacing the pseudodifferential arguments by  purely functional analytic
  ones. 
\end{remark}

\section{Mapping the positive scalar curvature sequence to analysis}

In this section, we finally prove Theorem \ref{theo:commute-r-pos}. One part
of this theorem is
the construction and commutativity of the following diagram
\eqref{eq:StolzToAna}. 
\begin{equation*}
  \begin{CD}
 @>>>   \Omega^{\spin}_{n+1} (B\Gamma) @>>>   R^{\spin}_{n+1}(B\Gamma) @>>> \Pos^{\spin}_n (B\Gamma) @>>> \Omega^{\spin}_n (B\Gamma) @>>>\\
  &&    @VV{\beta}V @VV{\Ind_\Gamma}V  @VV{\rho_\Gamma}V @VV{\beta}V\\
    @>>>   K_{n+1} ( B\Gamma) @>>> K_{n+1} ( C^*_r\Gamma) @>>>
    K_{n+1}(D^*_\Gamma)  @>>>  K_{n} ( B\Gamma) @>>>\\
    \end{CD}
\end{equation*}

First of all, we need to give a precise definition for the vertical homomorphisms.
Consider an element $[Y,f\colon Y\to B\Gamma,g_{\pa}]\in R^{{\rm
spin}}_{n+1} (B\Gamma)$. Let $g_Y$ be a Riemannian metric on
$Y$ extending $g_{\pa}$. We consider the Galois $\Gamma$-covering $W:= f^* E\Gamma$, 
endowed with the
lifted metric $g_W$. We consider $(W_\infty,g)$, the complete Riemannian
manifold with cylindrical ends associated to $W$.  
We wish to define $\Ind_\Gamma ([Y,f\colon Y\to B\Gamma,g_{\pa}])\in K_{n+1} (C^*_\Gamma)$;
to this end consider the relative coarse index class $\Ind^{{\rm rel}} (D_{W_\infty})\in
K_{n+1} (\relC{W_\infty}{W}^\Gamma )$ and its image  $\Ind (D_{W})\in
K_{n+1} (C^* (W)^\Gamma)$  through the canonical isomorphism
$K_{n+1} (\relC{W_\infty}{W}^\Gamma )\simeq K_{n+1} (C^* (W)^\Gamma)$.
We then consider the image of this class through the canonical
isomorphism $u_*\colon K_{n+1} (\relC{W_\infty}{W}^\Gamma )\simeq K_{n+1} (C^*_\Gamma)$ induced by the
 classifying map $u\colon W\to E\Gamma$. We have denoted this image by $\Ind_\Gamma (D_W)$,
 see \ref{gamma-index}. We set
 $$\Ind_\Gamma ([Y,f\colon Y\to B\Gamma,g_{\pa}]):= \Ind_\Gamma (D_W) \;\;\in\;\; K_{n+1} (C^*_\Gamma)\,.$$
 That this   index map $R^{\spin}_{n+1}(X) \xrightarrow{\Ind_\Gamma}  K_{n+1} ( C^*_r\Gamma)$
is well defined i.e.~that $\Ind_\Gamma(D_W)$ is bordism invariant, can be
  proved in many different ways. In future work we plan to give a treatment of
  bordism invariance in the spirit of coarse index theory. Alternatively,
  relying on 
  published previous work, it follows from the compatibility  between the coarse index class $\Ind^{{\rm rel}}  (D_{W_\infty})$
and the Mishchenko-Fomenko  
index class, either obtained on the associated manifold with cylindrical ends or 
\`a la Atiyah-Patodi-Singer, see \ref{prop:equality-aps}, then applying to the
Mishchenko-Fomenko index class \cite{Bunke} or \cite{LPPSC},
 where \cite{Bunke} employs a relative index theorem and
\cite{LPPSC} is based on a gluing formula for index classes.
  We remark that it would also be possible to state
and prove a relative index theorem similar to \cite{Bunke} but in coarse geometry and then apply
Bunke's argument directly to the coarse index class  $\Ind_\Gamma (D_W)$. \\
Consider now an element $[Z,f\colon Z\to B\Gamma,g_Z]\in \Pos^{\spin}_n (B\Gamma) $;
we consider the $\Gamma$-covering $M:=f^* B\Gamma$ and we endow it with 
the lifted metric $g$. Then, by definition,
 $$\rho_{\Gamma} [Z,f\colon Z\to B\Gamma,g_Z] = \rho_\Gamma (g)\;\;\in\;\; K_{n+1} (D^*_\Gamma)\,.$$
The fact that $\rho_\Gamma$ is well defined follows from Corollary \ref{corol:part_index_special}.
Finally, let us recall the definition of the  map $\beta
\colon \Omega^{\spin}_{n+1} (B\Gamma)\to K_{n+1} (B\Gamma)$, as given by Higson and Roe
in \cite{higson-roeI,higson-roeII,higson-roeIII}.
Consider an element  $[Z,f\colon Z\to B\Gamma]\in \Omega^{\spin}_{n+1} (B\Gamma)$  and let
$M:= f^* E\Gamma$, endowed with any $\Gamma$-invariant Riemannian  metric $g$. We consider the class 
$[D_M]\in K_{n} (D^* (M)^\Gamma/C^* (M)^\Gamma)\simeq K_{n+1} (M/\Gamma)\equiv K_{n+1} (Z)$
and we push it forward through $f_*$ to $K_{n+1} (B\Gamma)$:
$$\beta [Z,f\colon Z\to B\Gamma] := f_* [D_M]\;\in\; K_{n+1} (B\Gamma).$$

We must now tackle the commutativity of the diagram. We consider the three distinct squares
of the diagram from left to right. The commutativity of the first square, which  is implicitly discussed
in the work of Higson-Roe, follows from the definition of the $C^*_r \Gamma$-index class, as
given in Subsection \ref{subsect:ind-and-rho} and Subsection
\ref{sec:index_class_boundary}.
The commutativity of the second square is a direct consequence 
of our   APS index theorem, see Corollary 
\ref{corollary-main} and more precisely  formula 
\eqref{main-index-eq}.
The commutativity of the third square is again a direct consequence of 
the definitions.

The remaining part of Theorem \ref{theo:commute-r-pos} deals with a compact
space $X$ with fundamental group $\Gamma$ and universal covering
$\tilde 
X$. Here  one uses the canonical isomorphisms
$R^{\spin}_*(X)=R^{\spin}_*(\Gamma)$ (the 
structure groups depend only on the fundamental group),
$K_*(X)=K_{*+1}(D^*(\tilde X)^\Gamma/C^*(\tilde X)^\Gamma)$ and
$K_*(C^*(\tilde X))=K_*(C^*_r\Gamma)$. Then the proof of
\eqref{eq:StolzToAnaX} is exactly parallel to the proof of
\eqref{eq:StolzToAna} once we use the results explained in Remark \ref{remark:bordism-invariance}.
This finishes the proof of Theorem \ref{theo:commute-r-pos}.

\bibliography{surgery}
\bibliographystyle{plain}

\affiliationone{
   Paolo Piazza\\
   Sapienza Universit\`a di Roma\\ 
  Rome\\
   Italy 
   \email{piazza@mat.uniroma1.it}
}
\affiliationtwo{
   Thomas Schick\\
   Georg-August-Universit\"at G{\"o}ttingen\\
   Mathematisches Institut\\
  Bunsenstr.3, 37073 G\"ottingen\\
  Germany
   \email{schick@uni-math.gwdg.de}
}
\end{document}